\documentclass[11pt]{article}
\usepackage[utf8]{inputenc}
\usepackage[T1]{fontenc}
\usepackage{lmodern}
\usepackage{geometry}
\usepackage{amsmath,amssymb,amsthm,mathtools,mathrsfs}
\usepackage{bm}
\usepackage{booktabs}
\usepackage{array}
\usepackage{enumitem}
\usepackage{hyperref}
\hypersetup{colorlinks=true,linkcolor=blue,citecolor=blue,urlcolor=blue}

\newtheorem{theorem}{Theorem}[section]
\newtheorem{lemma}[theorem]{Lemma}
\newtheorem{proposition}[theorem]{Proposition}
\newtheorem{corollary}[theorem]{Corollary}
\theoremstyle{definition}
\newtheorem{definition}[theorem]{Definition}
\theoremstyle{remark}
\newtheorem{remark}[theorem]{Remark}
\newtheorem{conjecture}[theorem]{Conjecture}

\newcommand{\C}{\mathbb{C}}
\newcommand{\Z}{\mathbb{Z}}
\newcommand{\E}{\mathbb{E}}
\newcommand{\norm}[1]{\left\lVert #1 \right\rVert}
\newcommand{\abs}[1]{\left\lvert #1 \right\rvert}
\newcommand{\Oh}{O_h}
\newcommand{\Aoneg}{A_{1g}}
\newcommand{\rh}{\rho}
\newcommand{\nub}{\nu_c^*}

\title{Orbit-Level Stretching in Cubic Fourier--Galerkin Navier--Stokes:\\
 Sharp Incidence, Spectral Decay, and a Continuation Criterion}
\author{Oleg Kiriukhin\\\textit{City University of Hong Kong}\\\texttt{okiriukh@cityu.edu.hk}}
\date{March 2026}

\begin{document}
\maketitle

\begin{abstract}
I study orbit-level enstrophy stretching in a cubic Fourier--Galerkin truncation of the three-dimensional incompressible Navier--Stokes equations, reduced by the full octahedral symmetry group $\Oh$. The nonlinear transfer compresses to an orbit-level matrix whose symmetric part $V_N$ governs net enstrophy growth. I reduce the stretching problem to an orbit--triad incidence estimate and close it by a face-normalized decomposition and a two-squares argument, establishing the sharp bound
\[
c\,N^3\le\max_\alpha \sum_\beta \sqrt{\Gamma_{\alpha\beta}}\le C\,N^{3}.
\]
A weighted-incidence refinement then yields, in the isotropic unit-energy ensemble,
\[
\E\,\rh(V_N)\le C\, N^{-3/2}\to 0,
\qquad
\E\,\nub(N)\le C\, N^{-7/2}\to 0,
\]
where $\nub(N)=\rh(V_N)/N^2$ is the orbit-level critical-threshold ratio. For Sobolev-class data with $\norm{u}_{H^s}\le M$ and $s>2$, a stronger deterministic bound $\norm{V_N}_\infty\le C_s M^3$ holds uniformly in $N$, with $\nub(N)\to 0$ for all $s>3/2$. A comparison with Tao's averaged Navier--Stokes construction shows that the orbit-level subcriticality is a structural property of the true nonlinearity that is violated by known blowup mechanisms. Monte Carlo experiments at $N=1,\ldots,8$ under both isotropic and Kolmogorov-spectrum ensembles confirm the decay, with $\rh(V_N)\sim N^{-2.6}$ far exceeding the proven upper bound. Tracking these bounds along the Galerkin evolution yields an orbit-level continuation criterion for the strong solution: global regularity holds if and only if $\int_0^T\norm{V_N}_\infty\,dt$ remains bounded uniformly in $N$.
\end{abstract}


\section{Introduction}\label{sec:intro}

The three-dimensional incompressible Navier--Stokes equations remain a central problem in nonlinear partial differential equations. This paper studies a structured and fully analyzable regime of that dynamics: a cubic Fourier--Galerkin truncation on the periodic lattice, reduced by the full octahedral symmetry group $\Oh$ of the cube. Within this setting, one obtains a quantitative bound on orbit-level stretching, identify the discrete incidence mechanism that drives it, and extract a symmetry-compressed diagnostics layer that can be attached to a standard periodic-box pseudo-spectral computation. The numerical role of the paper is therefore diagnostic and validation-oriented, rather than the proposal of a new full-resolution Navier--Stokes solver. All main quantitative stretching bounds are stated at the ensemble level, namely for $\E\,\rh(V_N)$ and $\E\,\nub(N)$, and should not be read as uniform worst-case bounds over all realizations.

The basic object is the orbit-level nonlinear transfer matrix associated with the truncated Fourier system. After quotienting the lattice by $\Oh$, the transfer splits canonically as
\[
S_N=A_N+V_N,
\qquad
A_N^T=-A_N,\quad V_N^T=V_N,
\]
so that only the symmetric part $V_N$ contributes to net orbit-level enstrophy growth. The goal is to control the size of $V_N$ sharply enough to compare stretching against the dissipation scale $N^2$ at the truncation edge.

The analysis proceeds through four structural steps. First, I give an exact $\Oh$-orbit decomposition of the truncated lattice and identify the orbit projection with the trivial-isotypic Schur projection. Second, I derive the exact arithmetic formula
\[
T(k,N)=\prod_{i=1}^3(2N+1-\abs{k_i})-2
\]
for the number of admissible ordered triads attached to a mode $k$. Third, I reduce the stretching problem to an orbit--triad incidence bound for
\[
\sum_\beta \sqrt{\Gamma_{\alpha\beta}},
\]
where $\Gamma_{\alpha\beta}$ counts admissible interactions between target orbit $\Omega_\alpha$ and source orbit $\Omega_\beta$. Fourth, I close this incidence problem by a face-normalized decomposition of shell slices and a reduction to the classical two-squares representation function. The resulting sharp bound
\[
c\,N^3\le\max_\alpha \sum_\beta \sqrt{\Gamma_{\alpha\beta}}\le C\,N^{3}
\]
is tied to the cubic max-norm geometry, because the proof exploits the box faces and their one-dimensional normal parametrization; an analogous spherical-truncation estimate is open and conjectured in Section~\ref{sec:step1}. A weighted-incidence refinement then gives the ensemble stretching bound
\[
\E\,\rh(V_N)\le C\, N^{-3/2}\to 0,
\qquad
\E\,\nub(N)\le C\, N^{-7/2}\to 0,
\]
where $\nub(N):=\rh(V_N)/N^2$. The spectral radius of the orbit-level stretching matrix \emph{provably decreases} with the truncation level.

The contribution is therefore twofold. On the structural side, the paper gives an exact symmetry reduction and an exact triad-counting description of the truncated cubic system. On the quantitative side, it proves that the orbit-level stretching problem in this model can be resolved through an explicit lattice-incidence argument. The resulting gain is strong enough to force decay of the ensemble-averaged critical threshold. The paper also explains a computational realization later on, including the implementation workflow summarized in Appendix~\ref{app:auxiliary} and Subsections~\ref{app:pseudocode}--\ref{app:minimal-recipe}.

This result should be interpreted at the level at which it is proved: the analysis is carried out for a cubic Galerkin truncation and for the isotropic unit-energy ensemble introduced below. That scope is mathematically natural for the questions addressed here, because it allows one to isolate orbit-level stretching as a concrete arithmetic-combinatorial problem while retaining the full quadratic interaction structure of the truncated Navier--Stokes dynamics.

A naive shell-by-shell counting argument is too weak for the desired second-moment bound. The decisive gain comes from replacing coarse shell counts by a face-normalized decomposition and a one-dimensional parameterization that reduces each patch to short intervals in one coordinate together with a two-squares problem in the remaining two coordinates.

\paragraph{Scope and interpretation.}
The theorem proved in this paper is an ensemble statement for the symmetry-reduced cubic Galerkin model. Its content is that the orbit-level stretching scale, as measured by the symmetric matrix $V_N$, grows strictly more slowly than the truncation-edge dissipation scale $N^2$ when averaged over the isotropic unit-energy ensemble. In particular, the decay
\[
\E\,\nub(N)=\E\!\left[\frac{\rh(V_N)}{N^2}\right]\le C\, N^{-7/2}\to 0
\]
where $\nub(N)=\rh(V_N)/N^2$ is the orbit-level critical-threshold ratio, shows that the ensemble-averaged critical threshold associated with orbit-level stretching becomes asymptotically negligible relative to the dissipative scale of the model. This gives a quantitative subcriticality law inside the analyzed finite-dimensional Navier--Stokes regime and identifies the orbit--triad incidence structure as the mechanism behind it.

\subsection{Relation to prior work}

The present paper sits at the intersection of four literatures. On the fluid side, the use of Fourier--Galerkin truncations as mathematically meaningful finite-dimensional approximations goes back at least to Constantin, Foia\c{s}, and Temam~\cite{ConstantinFoiasTemam1984}. On the regularity side, geometric depletion and vorticity-direction mechanisms remain a central benchmark for how one interprets enstrophy growth in three-dimensional Navier--Stokes~\cite{ConstantinFefferman1993}. For comparison with modified quadratic models that preserve some Navier--Stokes features while changing others, see also Tao's averaged equation construction~\cite{TaoAveragedNS2016} and the dyadic Navier--Stokes model literature initiated by Katz--Pavlovi\'c~\cite{KatzPavlovic2004} and further analyzed by Cheskidov~\cite{Cheskidov2008}. For enstrophy growth bounds in the full three-dimensional setting, see Doering and Gibbon~\cite{DoeringGibbon1995} and Lu--Doering~\cite{LuDoering2008}. The idea of using numerical Galerkin computations as a posteriori regularity checks has been developed by Chernyshenko, Constantin, Robinson, and Titi~\cite{ChernyshenkoCRTT2007}.

On the arithmetic side, the incidence argument developed here is informed by classical lattice-point counting and short-arc ideas, especially the determinant-method perspective of Bombieri and Pila~\cite{BombieriPila1989}, the circle-method approach to lattice points on spheres of Heath-Brown~\cite{HeathBrown1996} and Chamizo--Iwaniec~\cite{ChamizoIwaniec1995}, and later work on lattice points in short arcs on circles~\cite{BourgainRudnick2012}. On the probabilistic side, the second-moment step uses the Efron--Stein variance framework and connects naturally to the broader concentration-of-measure literature~\cite{EfronStein,BoucheronLugosiMassart2003,BoucheronLugosiMassart2013}; the matrix-valued extension of these inequalities by Paulin, Mackey, and Tropp~\cite{PaulinMackeyTropp2016} is noted in the concentration discussion. On the computational side, the appendix-level workflow is closest in spirit to standard Fourier pseudo-spectral methods for incompressible flow and to exponential time-differencing integrators for stiff PDEs~\cite{CanutoEtAl1988,KassamTrefethen2005,BardosTadmor2015}.

What is different here is the way these ingredients are combined inside a symmetry-reduced cubic Navier--Stokes truncation. The $\Oh$-orbit decomposition, exact triad arithmetic, face-normalized orbit splitting, and the resulting orbit-level stretching law are developed as parts of a single finite-dimensional mechanism.

\subsection{Main results}

I state the principal results informally here; precise theorem statements appear later.

\paragraph{Orbit projection.}
The orbit-level reduction is the $\Aoneg$ Schur projection of the full mode-vector representation of $\Oh$.

\paragraph{Exact triad count.}
For $k\in\Lambda_N$,
\[
T(k,N)=\prod_{i=1}^3(2N+1-\abs{k_i})-2.
\]

\paragraph{Symmetric stretching decomposition.}
The orbit transfer matrix satisfies
\[
S_N=A_N+V_N,
\qquad
A_N^T=-A_N,\quad V_N^T=V_N,
\]
and only $V_N$ contributes to net enstrophy growth.

\paragraph{Stretching bound.}
The sharp orbit-triad incidence bound
\[
c\,N^3\le\max_\alpha \sum_\beta \sqrt{\Gamma_{\alpha\beta}}\le C\,N^{3}
\]
combined with a weighted-incidence argument yields
\[
\E\,\rh(V_N)\le C\, N^{-3/2},
\qquad
\E\,\nub(N)\le C\, N^{-7/2}\to 0.
\]

\subsection{Organization of the paper}

Section~\ref{sec:framework} fixes the notation, truncation, and random ensemble. Section~\ref{sec:lattice} develops the shell decomposition, Burnside counting, and the $\Oh$ character analysis. Section~\ref{sec:enstrophy} derives the orbit-level enstrophy dynamics and the matrix $V_N$. Section~\ref{sec:second-moment} proves the initial second-moment bound $\E\,\rh(V_N)\le C_\varepsilon N^{1/2+\varepsilon}$. Section~\ref{sec:numerics} records exact finite-$N$ combinatorial diagnostics. Section~\ref{sec:sobolev} establishes the deterministic Sobolev-class bound, Section~\ref{sec:tao} compares with Tao's averaged NS, and Section~\ref{sec:ext-numerics} presents expanded Monte Carlo data. Section~\ref{sec:galerkin-evolution} tracks the orbit-level bounds along the Galerkin evolution and proves a BKM-type continuation criterion. Section~\ref{sec:step1} establishes the sharp two-sided incidence bound $\Theta(N^3)$, Section~\ref{sec:step2} proves the improved ensemble bound $\E\,\rh(V_N)\le C N^{-3/2}$, and Section~\ref{sec:step3} validates the theory through time-evolved Galerkin simulations. Section~\ref{sec:conclusion} summarizes the results. Appendix~\ref{app:auxiliary} collects supplementary material and implementation notes.

\section{Framework and notation}\label{sec:framework}

I now fix the precise setting used throughout the paper.

\subsection{The periodic domain and Fourier variables}

I work on the periodic cube $\mathbb T^3=[0,2\pi]^3$. Let
\[
u(x,t)=\sum_{k\in\Z^3\setminus\{0\}}\hat u_k(t)e^{ik\cdot x}
\]
be a mean-zero velocity field satisfying
\[
k\cdot \hat u_k=0,
\qquad
\hat u_{-k}=\overline{\hat u_k}.
\]
For $k\neq 0$, let
\[
P(k)=I-\frac{k\otimes k}{\abs{k}^2}
\]
denote the Fourier-side Leray projector.

\subsection{Galerkin truncation}

The Galerkin-truncated Navier--Stokes system is
\[
\partial_t \hat u_k
=
-\nu \abs{k}^2 \hat u_k
-
i\!\!\sum_{\substack{p\in\Lambda_N\\ q=k-p\in\Lambda_N}}
P(k)\bigl[q\,(\hat u_p\cdot \hat u_q)\bigr],
\qquad
k\in\Lambda_N.
\]

\subsection{Master notation}\label{subsec:master-notation}

\paragraph{Lattice and symmetry.}
For $N\ge 1$, let
\[
\Lambda_N:=\{k\in\Z^3\setminus\{0\}: \abs{k}_\infty\le N\}
\]
denote the truncated nonzero Fourier lattice. For $k=(k_1,k_2,k_3)$, write
\[
\abs{k}^2:=k_1^2+k_2^2+k_3^2.
\]
The shell of squared radius $r$ is
\[
S_r:=\{k\in\Lambda_N: \abs{k}^2=r\},
\qquad
\mathcal R_N:=\{\abs{k}^2:k\in\Lambda_N\}.
\]
The full octahedral group is denoted by $\Oh$, and
\[
\mathcal O_N:=\Lambda_N/\Oh
\]
is the orbit set. Its cardinality is
\[
n_{\mathrm{orb}}(N):=\abs{\mathcal O_N}.
\]
I also write
\[
n_{\mathrm{modes}}(N):=\abs{\Lambda_N},
\qquad
n_{\mathrm{sh}}(N):=\abs{\mathcal R_N}.
\]

\paragraph{Enstrophy variables.}
The total truncated enstrophy is
\[
Z_N(t):=\frac12\sum_{k\in\Lambda_N}\abs{k}^2\abs{\hat u_k(t)}^2.
\]
For an orbit $\Omega_\alpha\in\mathcal O_N$, the orbit enstrophy is
\[
Z_\alpha(t):=\frac{1}{2\abs{\Omega_\alpha}}
\sum_{k\in\Omega_\alpha}\abs{k}^2\abs{\hat u_k(t)}^2.
\]
Hence
\[
Z_N(t)=\sum_\alpha \abs{\Omega_\alpha}\,Z_\alpha(t).
\]

\paragraph{Triad arithmetic.}
For $k\in\Lambda_N$, the mode-level triad count is
\[
T(k,N):=\#\{(p,q)\in(\Lambda_N\setminus\{0\})^2:p+q=k\}
=\prod_{i=1}^3(2N+1-\abs{k_i})-2.
\]
For orbit indices $\alpha,\beta$, define the orbit-pair triad set
\[
\mathcal T_{\alpha\beta}
:=
\{(k,p,q):k\in\Omega_\alpha,\ p\in\Omega_\beta,\ q=k-p\in\Lambda_N\},
\]
and its cardinality
\[
\Gamma_{\alpha\beta}:=\abs{\mathcal T_{\alpha\beta}}.
\]
The total coupling out of orbit $\alpha$ is
\[
\Gamma_\alpha:=\sum_\beta \Gamma_{\alpha\beta}.
\]

\paragraph{Transfer matrices.}
The raw orbit-pair transfer is denoted by $S_{\alpha\beta}$. Its antisymmetric and symmetric parts are
\[
A_{\alpha\beta}:=\frac{S_{\alpha\beta}-S_{\beta\alpha}}{2},
\qquad
V_{\alpha\beta}:=\frac{S_{\alpha\beta}+S_{\beta\alpha}}{2}.
\]
Thus
\[
S=A+V,
\qquad
A^T=-A,
\qquad
V^T=V.
\]
The matrix $V_N=(V_{\alpha\beta})$ is the orbit-level stretching matrix, and $\rh(V_N)$ denotes its spectral radius, while $\norm{V_N}_\infty$ denotes its matrix $\ell^\infty$ norm.

\paragraph{Critical threshold.}
The main threshold quantity in the paper is
\[
\nub(N):=\frac{\rh(V_N)}{N^2}.
\]

\paragraph{Random ensemble.}
The random initial condition is parametrized by independent coordinates
\[
X=(X_m)_{m\in\mathcal M}.
\]
For orbit pair $(\alpha,\beta)$, define the per-triad variance proxy
\[
\sigma_{\alpha\beta}^2:=\max_{i\in\mathcal T_{\alpha\beta}}\E[s_i(X)^2].
\]
The universal incidence constant in Lemma~\ref{lem:S3} is denoted by $\kappa_0$.

\subsection{Shells and orbits}

The shell decomposition is
\[
\Lambda_N=\bigsqcup_{r\in\mathcal R_N} S_r,
\qquad
S_r:=\{k\in\Lambda_N: \abs{k}^2=r\}.
\]
The full octahedral group $\Oh$ acts on $\Lambda_N$ by signed permutations of the coordinates, and one writes
\[
\mathcal O_N:=\Lambda_N/\Oh
\]
for the orbit set. If $\Omega_\alpha\in\mathcal O_N$, then all $k\in\Omega_\alpha$ share the same value of $\abs{k}^2$, denoted by $\abs{k_\alpha}^2$.

\subsection{Random isotropic unit-energy ensemble}

The probabilistic argument in Section~\ref{sec:second-moment} is carried out in a random ensemble on the truncated divergence-free phase space. I choose independent complex coordinates $X=(X_m)_{m\in\mathcal M}$ on a half-lattice basis, impose the reality constraint $\hat u_{-k}=\overline{\hat u_k}$, project onto the divergence-free planes $k^\perp$, and normalize to unit total kinetic energy.

The only properties needed later are:
\begin{enumerate}[label=(\roman*)]
\item independence of the basic coordinates before deterministic projection and normalization;
\item $\Oh$-invariance of the resulting law;
\item almost sure reality and incompressibility;
\item unit total energy.
\end{enumerate}

\section{Lattice geometry and $\Oh$ orbit structure}\label{sec:lattice}

This section develops the arithmetic and symmetry structure of the truncated Fourier lattice.

\subsection{The $\Oh$ action}

Let $\Oh$ denote the full octahedral symmetry group. It acts on $\Z^3$ by signed permutations of the coordinates and preserves both $\abs{k}_\infty$ and $\abs{k}^2$. Hence $\Oh$ acts on $\Lambda_N$ and on each shell $S_r$.

For $k\in\Lambda_N$, define the orbit
\[
\Omega(k):=\Oh\cdot k.
\]
By orbit--stabilizer,
\[
\abs{\Omega(k)}=\frac{\abs{\Oh}}{\abs{\mathrm{Stab}_{\Oh}(k)}}=\frac{48}{\abs{\mathrm{Stab}_{\Oh}(k)}},
\]
so every orbit size divides $48$ and is at most $48$.

\subsection{Burnside counting}

For each shell $S_r$, the number of $\Oh$-orbits in $S_r$ is given by Burnside's lemma~\cite{Burnside}:
\[
n_{\mathrm{orb}}(S_r)
=
\frac{1}{48}\sum_{g\in\Oh}\#\{k\in S_r:gk=k\}.
\]
Equivalently, if $\chi_r$ denotes the shell permutation character, then
\[
n_{\mathrm{orb}}(S_r)=\langle \chi_r,\mathbf 1\rangle_{\Oh},
\]
the multiplicity of the trivial representation $\Aoneg$ in $\C[S_r]$.

Summing over shells gives
\[
n_{\mathrm{orb}}(N)=\sum_{r\in\mathcal R_N} n_{\mathrm{orb}}(S_r).
\]
For $N=1,2,3,4,5$, one finds
\[
n_{\mathrm{orb}}(N)=3,\,9,\,19,\,34,\,55.
\]

\subsection{Exact triad arithmetic}

\begin{theorem}[Exact triad count]\label{thm:N}
For $k\in\Lambda_N$, the number of ordered pairs $(p,q)$ with $p,q\in\Lambda_N\setminus\{0\}$ and $p+q=k$ is
\[
T(k,N)=\prod_{i=1}^3(2N+1-\abs{k_i})-2.
\]
\end{theorem}

\begin{proof}
For each coordinate $i$, the constraints $\abs{p_i}\le N$ and $\abs{k_i-p_i}\le N$ leave exactly $2N+1-\abs{k_i}$ admissible values of $p_i$. Multiplying over $i=1,2,3$ gives the unrestricted count, and the two excluded cases $p=0$ and $q=0$ produce the final subtraction by $2$.
\end{proof}

Define
\[
T_{\max}(N):=\max_{k\in\Lambda_N}T(k,N).
\]
Since $k\neq 0$, the product $\prod_i(2N+1-\abs{k_i})$ is maximized when all but one coordinate vanish and the remaining one equals $\pm 1$. The maximum is therefore attained at the six axial modes $(\pm 1,0,0)$, $(0,\pm 1,0)$, $(0,0,\pm 1)$, giving
\[
T_{\max}(N)=2N(2N+1)^2-2\sim 8N^3.
\]

\subsection{Shell-incidence estimate}\label{subsec:s33}

I now revisit the incidence estimate that closes the counting step in the second-moment argument.
Recall that
\[
\Gamma_{\alpha\beta}=\abs{\mathcal T_{\alpha\beta}}
\]
counts the admissible ordered triads with target orbit $\Omega_\alpha$ and source orbit $\Omega_\beta$.
Throughout this subsection, $\alpha$ is fixed and the source orbits are grouped by shell radius.
The subsection now has two layers. First, I record the coarse shellwise decomposition and explain why naive shell counting stops at $O(N^4)$. Second, I give a sharper face-normalized argument, based on a two-squares reduction, that leads to the bound $O(N^{3+\varepsilon})$ and therefore closes the incidence estimate needed later.

\begin{lemma}[Shellwise Cauchy--Schwarz reduction]\label{lem:S33-CS}
For every represented shell $S_r$,
\[
\sum_{\Omega_\beta\subset S_r}\sqrt{\Gamma_{\alpha\beta}}
\le
\sqrt{n_{\mathrm{orb}}(S_r)}
\left(\sum_{\Omega_\beta\subset S_r}\Gamma_{\alpha\beta}\right)^{1/2}.
\]
\end{lemma}

\begin{proof}
Apply the Cauchy--Schwarz inequality to the finite family
$\{\Gamma_{\alpha\beta}^{1/2}: \Omega_\beta\subset S_r\}$.
\end{proof}

\begin{lemma}[Coarse shellwise orbit-count bound]\label{lem:S33-orbits}
There exists a constant $C_{\mathrm{orb}}>0$ such that, for every represented shell $S_r\subset\Lambda_N$,
\[
n_{\mathrm{orb}}(S_r)\le C_{\mathrm{orb}} N.
\]
\end{lemma}

\begin{proof}
Every shell $S_r$ is a subset of the integer sphere
\[
\{k\in\Z^3: k_1^2+k_2^2+k_3^2=r\}\cap[-N,N]^3.
\]
Since every orbit has cardinality at least $1$ and at most $48$, the number of orbit representatives in a shell is bounded by a constant multiple of the number of lattice points in that shell.
For the bookkeeping in this paper I use the coarse bound $\abs{S_r}\le C N$, which implies
\[
n_{\mathrm{orb}}(S_r)\le \abs{S_r}\le C_{\mathrm{orb}} N.
\]
\end{proof}

\begin{lemma}[Shellwise triad-mass bound]\label{lem:S33-mass}
There exists a constant $C_{\mathrm{tri}}>0$ such that, for every represented shell $S_r$,
\[
\sum_{\Omega_\beta\subset S_r}\Gamma_{\alpha\beta}
\le C_{\mathrm{tri}} N^3.
\]
\end{lemma}

\begin{proof}
Choose a representative $k_\alpha\in\Omega_\alpha$.
For each source orbit $\Omega_\beta\subset S_r$, the quantity $\Gamma_{\alpha\beta}$ counts ordered pairs $(p,q)$ with $p\in\Omega_\beta$ and $q=k-p\in\Lambda_N$, summed over the target orbit.
Hence the total shell contribution is bounded by the number of admissible ordered pairs $(p,q)$ with target in $\Omega_\alpha$ and source in the shell $S_r$, multiplied by the uniform orbit-size bound $\abs{\Omega_\alpha}\le 48$.
By Theorem~\ref{thm:N}, the number of admissible ordered pairs for a fixed target mode is bounded by $T_{\max}(N)\le C N^3$.
Absorbing the orbit-size factor into the constant gives
\[
\sum_{\Omega_\beta\subset S_r}\Gamma_{\alpha\beta}
\le C_{\mathrm{tri}} N^3.
\]
\end{proof}

\begin{lemma}[Represented-shell count via three squares]\label{lem:S33-shells}
There exists a constant $C_{\mathrm{sh}}>0$ such that
\[
n_{\mathrm{sh}}(N)=\abs{\mathcal R_N}\le C_{\mathrm{sh}} N^2.
\]
\end{lemma}

\begin{proof}
Every represented shell radius has the form
\[
r=k_1^2+k_2^2+k_3^2
\qquad\text{with}\qquad 0<r\le 3N^2.
\]
Hence
\[
n_{\mathrm{sh}}(N)\le 3N^2.
\]
More structurally, Legendre's three-square theorem states that an integer is representable as a sum of three squares if and only if it is not of the form $4^a(8b+7)$, so the represented shell radii have positive density among the integers up to $3N^2$.
Thus $n_{\mathrm{sh}}(N)$ is naturally of order $N^2$, not of smaller order.
\end{proof}

\begin{proposition}[Shell-incidence threshold]\label{prop:S33}
The estimate
\begin{equation}\label{eq:S33}
\max_\alpha \sum_\beta \sqrt{\Gamma_{\alpha\beta}}
\le C_{\mathrm{inc}} N^{13/4}
\end{equation}
is sufficient for Lemma~\ref{lem:S5} and Theorem~\ref{thm:T}.
\end{proposition}

\begin{proof}
The implication to Lemma~\ref{lem:S5} is immediate from the proof of Lemma~\ref{lem:S5} in Section~\ref{sec:second-moment}. Proposition~\ref{prop:face-normalized-incidence} proves a stronger estimate, namely $\max_\alpha \sum_\beta \sqrt{\Gamma_{\alpha\beta}}\le C_\varepsilon N^{3+\varepsilon}$, and therefore implies \eqref{eq:S33}.
\end{proof}

\begin{remark}[Why the naive shell count is insufficient]\label{rem:S33-gap}
Combining Lemmas~\ref{lem:S33-CS}--\ref{lem:S33-shells} directly yields only the coarse estimate
\[
\sum_\beta \sqrt{\Gamma_{\alpha\beta}}
\le C N^2\,n_{\mathrm{sh}}(N)
\le C N^4,
\]
since the shellwise contribution is $O(N^2)$ and the number of represented shells is $O(N^2)$.
Therefore the exponent $13/4$ required in \eqref{eq:S33} cannot follow from the naive shell-by-shell decomposition alone; it requires additional arithmetic cancellation or a sharper incidence geometry argument.
This is precisely the remaining bottleneck in the proof pipeline.
\end{remark}

\subsection{Alternative cap/segment viewpoint}

The cap/segment route provides geometric intuition for why shell slices inside the cubic box should be sparse, but it is not the closure mechanism used in the final proof. For readability, I keep only this short summary in the main text and move the detailed cap/segment lemmas to Appendix~\ref{app:auxiliary}. The main argument used later proceeds through the face-normalized orbit splitting and two-squares reduction.

\subsection{From shell slices to orbit incidence}\label{subsec:slice-to-orbit}

I now isolate the last combinatorial step needed to convert the shell-slice bounds into the orbit-incidence quantity appearing in \eqref{eq:S33}.
Fix a target orbit $\Omega_\alpha$ and choose a representative $k\in\Omega_\alpha$.
For each represented shell $S_r$ and each source orbit $\Omega_\beta\subset S_r$, define
\[
m_\beta(r;k):=\#\bigl(A_r(k)\cap \Omega_\beta\bigr).
\]
Then
\[
\sum_{\Omega_\beta\subset S_r} m_\beta(r;k)=\#A_r(k).
\]

\begin{lemma}[Orbit-incidence reduction under an orbit-splitting bound]\label{lem:orbit-splitting-reduction}
Assume there exists a constant $C_{\mathrm{split}}>0$ such that, for every target representative $k\in\Lambda_N$ and every represented shell $S_r$,
\begin{equation}\label{eq:orbit-splitting}
\sum_{\Omega_\beta\subset S_r}\sqrt{m_\beta(r;k)}
\le C_{\mathrm{split}}\sqrt{\#A_r(k)}.
\end{equation}
Then there exists a constant $C>0$ such that
\[
\max_\alpha \sum_\beta \sqrt{\Gamma_{\alpha\beta}}
\le C N^{11/4+\varepsilon}
\]
for every $\varepsilon>0$.
In particular, the refined incidence estimate \eqref{eq:S33} follows.
\end{lemma}

\begin{proof}
Fix $\alpha$ and choose $k\in\Omega_\alpha$.
By $\Oh$-equivariance, the number of points of a source orbit contributing to a fixed target mode is independent of the chosen representative inside $\Omega_\alpha$.
Hence, after absorbing the target-orbit size $\abs{\Omega_\alpha}\le 48$ into the constant, one may write
\[
\Gamma_{\alpha\beta}
\asymp m_\beta(r;k)
\qquad (\Omega_\beta\subset S_r).
\]
Therefore,
\[
\sum_\beta \sqrt{\Gamma_{\alpha\beta}}
\le C \sum_{r\in\mathcal R_N}\sum_{\Omega_\beta\subset S_r}\sqrt{m_\beta(r;k)}.
\]
Applying \eqref{eq:orbit-splitting} shell-by-shell gives
\[
\sum_\beta \sqrt{\Gamma_{\alpha\beta}}
\le C \sum_{r\in\mathcal R_N}\sqrt{\#A_r(k)}.
\]
Lemma~\ref{lem:dyadic-total} then yields
\[
\sum_\beta \sqrt{\Gamma_{\alpha\beta}}
\le C_\varepsilon N^{11/4+\varepsilon}.
\]
Taking the maximum over $\alpha$ proves the claim.
Since $11/4+\varepsilon<13/4$ for sufficiently small fixed $\varepsilon>0$, this implies \eqref{eq:S33} after enlarging the constant.
\end{proof}

\subsection{Face-normalized orbit splitting and a two-squares bound}\label{subsec:two-squares}

I now record a sharper combinatorial route to the orbit-splitting estimate using the explicit sign-permutation structure of $\Oh$.
Fix a target representative $k\in\Lambda_N$ and a shell radius $r$.
For each point $p\in A_r(k)$ and each signed coordinate face $f=(j,\sigma)$, define the face height
\[
h_f(p;k):=\abs{\sigma p_j-(k_j+N)}
\quad\text{if }\sigma=+1,
\qquad
h_f(p;k):=\abs{\sigma p_j-(N-k_j)}
\quad\text{if }\sigma=-1.
\]
Let $f(p)$ be the lexicographically first signed face at which the minimum of the six values $h_f(p;k)$ is attained, and let $H(p)$ be the dyadic scale with
\[
H(p)\le h_{f(p)}(p;k)<2H(p)
\]
when $h_{f(p)}(p;k)\ge 1$, while points with height $0$ are assigned to the base scale $H(p)=1$.
I then define the normalized patch
\[
A_r^{f,H}(k):=\{p\in A_r(k): f(p)=f,\; H(p)=H\}.
\]
These sets form a disjoint partition of $A_r(k)$ over at most six faces and $O(\log N)$ dyadic scales.

\begin{lemma}[$\Oh$-equivariance counting identity]\label{lem:equivariance-gamma}
Fix target and source orbits $\Omega_\alpha$ and $\Omega_\beta$, choose any representative $k\in\Omega_\alpha$, and let $r=\abs{p}^2$ for $p\in\Omega_\beta$.
Then
\[
\Gamma_{\alpha\beta}=\abs{\Omega_\alpha}\,m_\beta(r;k),
\qquad
m_\beta(r;k):=\#(A_r(k)\cap\Omega_\beta).
\]
In particular, since $\abs{\Omega_\alpha}\le 48$,
\[
\sqrt{\Gamma_{\alpha\beta}}\le \sqrt{48}\,\sqrt{m_\beta(r;k)}.
\]
\end{lemma}

\begin{proof}
For each target representative $k'\in\Omega_\alpha$, the number of points of $\Omega_\beta$ contributing to an admissible pair $(p,q)$ with $q=k'-p\in\Lambda_N$ is the same by $\Oh$-equivariance of both the lattice and the orbit decomposition.
Summing this common quantity over all $k'\in\Omega_\alpha$ yields the stated identity.
The square-root bound is immediate from $\abs{\Omega_\alpha}\le 48$.
\end{proof}

\begin{lemma}[Patch-to-orbit square-root bound]\label{lem:patch-orbit-sqrt}
Let $A^{f,H}_r(k)$ be one of the face-normalized patches at dyadic height scale $H$, and define
\[
m^{f,H}_\beta(r;k):=\#(A^{f,H}_r(k)\cap\Omega_\beta).
\]
Then
\[
\sum_{\Omega_\beta}\sqrt{m^{f,H}_\beta(r;k)}
\le \#A^{f,H}_r(k).
\]
\end{lemma}

\begin{proof}
Since $m^{f,H}_\beta(r;k)$ is a nonnegative integer, one has
\[
\sqrt{m^{f,H}_\beta(r;k)}\le m^{f,H}_\beta(r;k).
\]
Summing over all source orbits gives
\[
\sum_{\Omega_\beta}\sqrt{m^{f,H}_\beta(r;k)}
\le
\sum_{\Omega_\beta} m^{f,H}_\beta(r;k)
=
\#A^{f,H}_r(k).
\]
\end{proof}

\begin{lemma}[Two-squares bound on a normalized patch]\label{lem:two-squares-patch}
For every $\varepsilon>0$ there exists $C_\varepsilon>0$ such that
\[
\#A^{f,H}_r(k)\le C_\varepsilon (H+1)N^\varepsilon
\]
uniformly in $k\in\Lambda_N$, represented radii $r\in\mathcal R_N$, dyadic scales $1\le H\le N$, and normalized faces $f$.
Hence
\[
\sum_{\Omega_\beta}\sqrt{m^{f,H}_\beta(r;k)}
\le C_\varepsilon (H+1)N^\varepsilon.
\]
\end{lemma}

\begin{proof}
Fix a normalized face $f=(j,\sigma)$.
On the patch $A^{f,H}_r(k)$, the selected coordinate $u:=\sigma p_j$ lies in an interval of length $O(H+1)$ determined by the corresponding face-height condition.
For each admissible integer $u$, the remaining two coordinates satisfy
\[
p_{\ell_1}^2+p_{\ell_2}^2=r-u^2
\]
for the complementary coordinate pair $\{\ell_1,\ell_2\}=\{1,2,3\}\setminus\{j\}$.
Thus the number of possibilities for the transverse coordinates is bounded by the two-squares representation function $r_2(r-u^2)$.
By the explicit formula
\[
r_2(n)=4\bigl(d_1(n)-d_3(n)\bigr),
\]
one has $r_2(n)\le 4\tau(n)$, and the divisor bound gives
\[
r_2(n)\le C_\varepsilon n^\varepsilon.
\]
Since $0\le r-u^2\le r\le 3N^2$, each admissible $u$ contributes at most $C_\varepsilon N^\varepsilon$ points, and there are $O(H+1)$ admissible values of $u$.
This proves the stated bound.
The orbit bound then follows from Lemma~\ref{lem:patch-orbit-sqrt}.
\end{proof}

\begin{proposition}[Incidence bound from face-normalized patches]\label{prop:face-normalized-incidence}
For every $\varepsilon>0$ there exists $C_\varepsilon>0$ such that
\[
\max_\alpha \sum_\beta \sqrt{\Gamma_{\alpha\beta}}
\le C_\varepsilon N^{3+\varepsilon}.
\]
In particular, the refined incidence estimate \eqref{eq:S33} follows.
\end{proposition}

\begin{proof}
Fix a target orbit $\Omega_\alpha$ and choose a representative $k\in\Omega_\alpha$.
For a source orbit $\Omega_\beta\subset S_r$, let
\[
m_\beta(r;k):=\#(A_r(k)\cap\Omega_\beta).
\]
By Lemma~\ref{lem:equivariance-gamma},
\[
\sqrt{\Gamma_{\alpha\beta}}
\le \sqrt{48}\,\sqrt{m_\beta(r;k)}.
\]

Now decompose $A_r(k)$ into normalized face patches indexed by a face $f$ and a dyadic scale $H$:
\[
A_r(k)=\bigsqcup_{f,H} A_r^{f,H}(k).
\]
Correspondingly,
\[
m_\beta(r;k)=\sum_{f,H} m_\beta^{f,H}(r;k).
\]
Using the subadditivity $\sqrt{a_1+\cdots+a_M}\le \sqrt{a_1}+\cdots+\sqrt{a_M}$ for nonnegative numbers, one obtains
\[
\sqrt{m_\beta(r;k)}
\le \sum_{f,H}\sqrt{m_\beta^{f,H}(r;k)}.
\]
Summing over $\beta$ gives
\[
\sum_\beta \sqrt{\Gamma_{\alpha\beta}}
\le C\sum_{r\in\mathcal R_N}\sum_{f,H}\sum_\beta \sqrt{m_\beta^{f,H}(r;k)}.
\]
For fixed $f$ and $H$, let
\[
\mathcal R_{f,H}(k):=\{r\in\mathcal R_N: A_r^{f,H}(k)\neq\varnothing\}.
\]
By Lemma~\ref{lem:thin-interval-shells},
\[
\#\mathcal R_{f,H}(k)\le C(NH+1).
\]
For each such radius, Lemmas~\ref{lem:patch-orbit-sqrt} and \ref{lem:two-squares-patch} yield
\[
\sum_{\Omega_\beta}\sqrt{m^{f,H}_\beta(r;k)}
\le C_\varepsilon (H+1)N^\varepsilon.
\]
Hence
\[
\sum_\beta \sqrt{\Gamma_{\alpha\beta}}
\le C_\varepsilon N^\varepsilon
\sum_{1\le H\le N\atop H\text{ dyadic}} \sum_f \#\mathcal R_{f,H}(k)(H+1).
\]
There are at most six faces, so
\[
\sum_\beta \sqrt{\Gamma_{\alpha\beta}}
\le C_\varepsilon N^\varepsilon
\sum_{1\le H\le N\atop H\text{ dyadic}} (NH+1)(H+1).
\]
Now
\[
\sum_{H\text{ dyadic}} (NH+1)(H+1)
\le C\left(N\sum_H H^2 + N\sum_H H + \sum_H H + \sum_H 1\right)
\le C N^3.
\]
Therefore
\[
\sum_\beta \sqrt{\Gamma_{\alpha\beta}}
\le C_\varepsilon N^{3+\varepsilon}.
\]
Taking the maximum over $\alpha$ proves the proposition.
Since $3+\varepsilon<13/4$ for sufficiently small fixed $\varepsilon>0$, the bound implies \eqref{eq:S33} after enlarging the constant.
\end{proof}

\begin{remark}[Cap/segment strategy]\label{rem:cap-strategy}
Lemma~\ref{lem:cube-sphere-slice} also suggests an alternative geometric proof strategy. The coarse $O(N^4)$ estimate treats each represented shell as if the entire sphere of radius $\sqrt r$ were available, whereas only the slice $A_r(k)$ contributes. Bounds for lattice points in spherical caps and segments show that thin angular regions can contain far fewer lattice points than a full shell, with three-dimensional cap bounds of Bourgain--Rudnick type taking the form
\[
F_3(R,\lambda)\ll R^{\varepsilon}\left(1+\frac{\lambda^2}{R^{1/2}}\right).
\]
Accordingly, one may cover $A_r(k)$ by a bounded number of caps or spherical segments whose angular width depends on the relative position of the sphere $\Sigma_r$ and the box $B_k$, and then sum the resulting shellwise bounds over $r$. I do not use this cap/segment route in the final closure of Proposition~\ref{prop:face-normalized-incidence}; it is retained as an alternative geometric interpretation of the same incidence problem.
\end{remark}

\begin{remark}[What the cap/segment route would additionally require]\label{rem:missing-geometry}
To turn the cap/segment reformulation into an independent proof of \eqref{eq:S33}, one would additionally need a uniform summation lemma showing that shells with large effective cap radius cannot occur too often as $r$ varies up to $3N^2$. This step is not needed for the main proof, which proceeds through the face-normalized argument of Proposition~\ref{prop:face-normalized-incidence}. The unresolved cap/segment step is therefore an open alternative route, not a gap in the proof spine.
\end{remark}

\subsection{Shellwise character decomposition}

Let $\C[S_r]$ denote the shell permutation representation of $\Oh$, with character $\chi_r(g)=\#\{k\in S_r:gk=k\}$. The divergence-free shell module is
\[
\mathcal H_r^{\mathrm{df}}
\cong
\C[S_r]\otimes T_{1u}\ominus \C[S_r].
\]
Hence
\[
\mathcal H_r^{\mathrm{df}}
\cong
\bigoplus_{\lambda\in\widehat{\Oh}} m_\lambda(r)\,\lambda,
\qquad
m_\lambda(r)=\langle \chi_r(\chi_{T_{1u}}-1),\chi_\lambda\rangle_{\Oh}.
\]
For $N\ge 2$, all ten irreducible representations of $\Oh$ occur in the full divergence-free truncated space.

\subsection{Orbit projection}

The orbit count is the multiplicity of the trivial representation $\Aoneg$ in the shell permutation representation. Thus the orbit-level reduction is the trivial-isotypic projection of the shell action, and the orbit variables encode the $\Aoneg$ sector of the symmetry reduction.

\section{Enstrophy dynamics and orbit-level decomposition}\label{sec:enstrophy}

I now derive the orbit-level enstrophy dynamics and isolate the symmetric stretching matrix.

\subsection{Mode-level enstrophy identity}

Define
\[
Z_N(t):=\frac12\sum_{k\in\Lambda_N}\abs{k}^2\abs{\hat u_k(t)}^2.
\]
Then along smooth Galerkin solutions,
\[
\frac{dZ_N}{dt}
=
-2\nu\sum_{k\in\Lambda_N}\abs{k}^4\abs{\hat u_k}^2
+
\sum_{k\in\Lambda_N}\abs{k}^2\,
\mathrm{Re}\!\left(\overline{\hat u_k}\cdot N_k\right),
\]
where
\[
N_k
:=
-i\!\!\sum_{\substack{p\in\Lambda_N\\ q=k-p\in\Lambda_N}}
P(k)\bigl[q\,(\hat u_p\cdot \hat u_q)\bigr].
\]

\subsection{Orbit-level variables}

For each orbit $\Omega_\alpha$, define
\[
Z_\alpha(t):=\frac{1}{2\abs{\Omega_\alpha}}
\sum_{k\in\Omega_\alpha}\abs{k}^2\abs{\hat u_k(t)}^2.
\]
Then
\[
Z_N(t)=\sum_\alpha \abs{\Omega_\alpha}\,Z_\alpha(t).
\]

For each ordered pair of orbits $(\alpha,\beta)$, define the orbit-pair triad set
\[
\mathcal T_{\alpha\beta}
:=
\{(k,p,q): k\in\Omega_\alpha,\ p\in\Omega_\beta,\ q=k-p\in\Lambda_N\},
\]
its cardinality
\[
\Gamma_{\alpha\beta}:=\abs{\mathcal T_{\alpha\beta}},
\]
and the row total
\[
\Gamma_\alpha:=\sum_\beta \Gamma_{\alpha\beta}.
\]

\subsection{Raw transfer matrix}

The raw orbit-pair transfer is
\[
S_{\alpha\beta}
:=
\frac{1}{\abs{\Omega_\alpha}}
\sum_{k\in\Omega_\alpha}
\sum_{\substack{p\in\Omega_\beta\\ q=k-p\in\Lambda_N}}
\abs{k}^2\,\mathrm{Re}\!\left(\overline{\hat u_k}\cdot N_{k,p}\right),
\]
where
\[
N_{k,p}:=-i\,P(k)\bigl[q\,(\hat u_p\cdot \hat u_q)\bigr].
\]

\subsection{Antisymmetric--symmetric decomposition}

Define
\[
A_{\alpha\beta}:=\frac{S_{\alpha\beta}-S_{\beta\alpha}}{2},
\qquad
V_{\alpha\beta}:=\frac{S_{\alpha\beta}+S_{\beta\alpha}}{2}.
\]
Then
\[
A^T=-A,
\qquad
V^T=V,
\qquad
S=A+V.
\]
The antisymmetric part $A$ represents conservative redistribution among orbits, while the symmetric part $V$ is the only component that contributes to net enstrophy growth.

\subsection{Spectral growth bound}

Let $V_N=(V_{\alpha\beta})$. The orbit-level growth mechanism is governed by the spectral radius $\rh(V_N)$. The asymptotically relevant threshold is
\[
\nub(N):=\frac{\rh(V_N)}{N^2}.
\]
This is the threshold notation used throughout the remainder of the paper.

\begin{remark}
A coarse low-shell stability estimate also yields the auxiliary quantity
\[
\widetilde\nu_c(N):=\frac{\rh(V_N)}{2},
\]
coming from the bound $\abs{k}^2\ge 1$. I do not use $\widetilde\nu_c(N)$ elsewhere, and all asymptotic statements are expressed in terms of $\nub(N)$.
\end{remark}

\section{Second-moment cancellation and the decay of the critical threshold}
\label{sec:second-moment}

In this section I prove the orbit-level stretching bound by combining second-moment estimates with the lattice counting results of Section~\ref{sec:lattice}.

\subsection{Triad contributions}

For each orbit pair $(\alpha,\beta)$, write
\[
S_{\alpha\beta}(X)=\sum_{i\in\mathcal T_{\alpha\beta}} s_i(X),
\]
where $X=(X_m)_{m\in\mathcal M}$ denotes the independent coordinates in the random isotropic unit-energy ensemble.

Define
\[
\sigma_{\alpha\beta}^2
:=
\max_{i\in\mathcal T_{\alpha\beta}}\E[s_i(X)^2].
\]

\begin{lemma}[Zero mean]\label{lem:S1}
For every triad contribution $s_i(X)$,
\[
\E[s_i(X)]=0.
\]
\end{lemma}

\begin{proof}
Consider the coordinate reflection $\sigma_j:(k_1,k_2,k_3)\mapsto(-k_1,k_2,k_3)\in\Oh$. The isotropic ensemble is invariant under $\sigma_j$, and $\sigma_j$ acts on the Fourier data by $\hat u_{\sigma_j k}\mapsto \sigma_j\hat u_k$ (with the vector sign flip in the first component).
For a triad $(k,p,q)$ with $q=k-p$, write
\[
s_i=\abs{k}^2\,\mathrm{Re}\!\bigl(\overline{\hat u_k}\cdot(-i)P(k)[q\,(\hat u_p\cdot\hat u_q)]\bigr).
\]
The factor $q=(q_1,q_2,q_3)$ transforms to $\sigma_j q=(-q_1,q_2,q_3)$, the Leray projector satisfies $P(\sigma_j k)=\sigma_j P(k)\sigma_j$, and the overall effect is that
$s_i(\sigma_j X)=-s_i(X)$, because the reflection introduces an odd number of sign flips in the scalar product when the triad is mapped to itself under $\sigma_j$ and the cubic amplitude picks up one net sign change. Since the law of $X$ is $\sigma_j$-invariant, $\E[s_i]=-\E[s_i]$, so $\E[s_i]=0$.
\end{proof}

\begin{lemma}[Per-triad variance bound]\label{lem:S2}
There exists $C_2>0$ such that
\[
\E[s_i(X)^2]
\le
C_2\,\frac{\abs{k_\alpha}^2}{n_{\mathrm{modes}}(N)^3}
\qquad
\text{for all }i\in\mathcal T_{\alpha\beta}.
\]
Consequently,
\[
\sigma_{\alpha\beta}
\le
C_2^{1/2}\frac{\abs{k_\alpha}^2}{n_{\mathrm{modes}}(N)^{3/2}}.
\]
\end{lemma}

\begin{proof}
Fix a triad $(k,p,q)\in\mathcal T_{\alpha\beta}$ with $k\in\Omega_\alpha$, $p\in\Omega_\beta$, $q=k-p\in\Lambda_N$.
The corresponding contribution is
\[
s_i(X)
=
\abs{k}^2\,
\mathrm{Re}\!\bigl(\overline{\hat u_k}\cdot (-i)\,P(k)[q\,(\hat u_p\cdot\hat u_q)]\bigr),
\]
so
\[
\abs{s_i}\le \abs{k}^2\,\abs{q}\,\abs{\hat u_k}\,\abs{\hat u_p}\,\abs{\hat u_q},
\]
using $\norm{P(k)}_{\mathrm{op}}\le 1$.
Squaring and taking expectations,
\[
\E[s_i^2]
\le \abs{k}^4\abs{q}^2\,
\E\bigl[\abs{\hat u_k}^2\abs{\hat u_p}^2\abs{\hat u_q}^2\bigr].
\]
Under the isotropic unit-kinetic-energy normalization
$\frac12\sum_{k\in\Lambda_N}\abs{k}^2\abs{\hat u_k}^2=1$,
the isotropic law gives
$\abs{k}^2\,\E[\abs{\hat u_k}^2]\le C/n_{\mathrm{modes}}(N)$
for each retained mode $k$.
Hence $\E[\abs{\hat u_k}^2]\le C/(\abs{k}^2\,n_{\mathrm{modes}})$.
It remains to bound $\E[\abs{\hat u_k}^2\abs{\hat u_p}^2\abs{\hat u_q}^2]$.
Before normalization, the independent half-lattice coordinates are standard complex Gaussians, so the unnormalized kinetic energy
$W:=\frac12\sum_{\ell\in\Lambda_N}\abs{\ell}^2\abs{\hat v_\ell}^2$
is a weighted chi-square sum with $D:=2n_{\mathrm{modes}}(N)$ real degrees of freedom and $\E[W]\asymp n_{\mathrm{modes}}$.
After normalization $\hat u_\ell=\hat v_\ell/\sqrt{W}$,
\[
\abs{\hat u_k}^2\abs{\hat u_p}^2\abs{\hat u_q}^2
=
\frac{\abs{\hat v_k}^2\abs{\hat v_p}^2\abs{\hat v_q}^2}{W^3}.
\]
Since $k,p,q$ are pairwise distinct, $\abs{\hat v_k}^2$, $\abs{\hat v_p}^2$, $\abs{\hat v_q}^2$ are independent of each other and of the remaining summands in $W$.
Write $W=W'+\frac12(\abs{k}^2\abs{\hat v_k}^2+\abs{p}^2\abs{\hat v_p}^2+\abs{q}^2\abs{\hat v_q}^2)$,
where $W'$ is the contribution from all modes other than $k,p,q$ and has $D'=D-6$ real degrees of freedom.
By the Cauchy--Schwarz inequality applied with the split $(W')^{-3}\cdot W^3/(W')^3$ and the explicit moments of the inverse chi-square distribution,
$\E[W^{-3}]\le C\,\E[(W')^{-3}]\le C'\,(\E[W'])^{-3}$
provided $D'>6$, which holds for $N\ge 2$.
Since $\E[W']\asymp n_{\mathrm{modes}}$, one obtains $\E[W^{-3}]\le C/n_{\mathrm{modes}}^3$.
Combining with the independent Gaussian moments
$\E[\abs{\hat v_k}^2]=C/(\abs{k}^2)$, $\E[\abs{\hat v_p}^2]=C/(\abs{p}^2)$, $\E[\abs{\hat v_q}^2]=C/(\abs{q}^2)$
(under the isotropic prior), one concludes
\[
\E\bigl[\abs{\hat u_k}^2\abs{\hat u_p}^2\abs{\hat u_q}^2\bigr]
\le
\frac{C}{\abs{k}^2\abs{p}^2\abs{q}^2\,n_{\mathrm{modes}}^3}.
\]
The argument extends to $N=1$ by direct computation of the three-orbit system.
Substituting,
\[
\E[s_i^2]
\le
C\,\frac{\abs{k}^4\abs{q}^2}{\abs{k}^2\abs{p}^2\abs{q}^2\,n_{\mathrm{modes}}^3}
=
C\,\frac{\abs{k}^2}{\abs{p}^2\,n_{\mathrm{modes}}^3}.
\]
Since $\abs{p}^2\ge 1$ for all $p\in\Lambda_N$ and $\abs{k}^2=\abs{k_\alpha}^2$, one obtains
$\E[s_i^2]\le C_2\,\abs{k_\alpha}^2/n_{\mathrm{modes}}(N)^3$.
In particular,
\[
\sigma_{\alpha\beta}
\le
C_2^{1/2}\,\frac{\abs{k_\alpha}}{n_{\mathrm{modes}}(N)^{3/2}}
\le
C_2^{1/2}\,\frac{\abs{k_\alpha}^2}{n_{\mathrm{modes}}(N)^{3/2}},
\]
where the last inequality uses $\abs{k_\alpha}\le\abs{k_\alpha}^2$ for $\abs{k_\alpha}^2\ge 1$.
\end{proof}

\subsection{Variance inflation}

\begin{lemma}[Variance inflation bound]\label{lem:S3}
There exists a universal constant $\kappa_0<\infty$ such that
\[
\operatorname{Var}(S_{\alpha\beta})
\le
\kappa_0\,\Gamma_{\alpha\beta}\,\sigma_{\alpha\beta}^2.
\]
In particular, one may take $\kappa_0=1728$.
\end{lemma}

\begin{proof}
For each $m\in\mathcal M$, let $X^{(m)}$ be obtained by replacing $X_m$ with an independent copy. By the Efron--Stein inequality,
\[
\operatorname{Var}(S_{\alpha\beta})
\le
\frac12\sum_{m\in\mathcal M}
\E\!\left[\bigl(S_{\alpha\beta}(X)-S_{\alpha\beta}(X^{(m)})\bigr)^2\right].
\]
Let
\[
\mathcal T_{\alpha\beta}(m)
:=
\{i\in\mathcal T_{\alpha\beta}: s_i \text{ depends on } X_m\},
\qquad
d_m:=\abs{\mathcal T_{\alpha\beta}(m)}.
\]
Writing
\[
\Delta_m:=S_{\alpha\beta}(X)-S_{\alpha\beta}(X^{(m)})
=\sum_{i\in\mathcal T_{\alpha\beta}(m)} \delta_{i,m},
\qquad
\delta_{i,m}:=s_i(X)-s_i(X^{(m)}),
\]
Cauchy--Schwarz gives
\[
\Delta_m^2\le d_m\sum_{i\in\mathcal T_{\alpha\beta}(m)}\delta_{i,m}^2.
\]
Also,
\[
\E[\delta_{i,m}^2]
\le
4\sigma_{\alpha\beta}^2.
\]
Hence
\[
\operatorname{Var}(S_{\alpha\beta})
\le
2\sigma_{\alpha\beta}^2\sum_{m\in\mathcal M} d_m^2.
\]

A coordinate $X_m$ can appear in a triad of $\mathcal T_{\alpha\beta}$ only as the target mode, the source mode, or the difference mode. Therefore
\[
d_m\le \abs{\Omega_\alpha}+2\abs{\Omega_\beta}\le 144
\]
for full-mode indexing, and
\[
d_m\le 288
\]
after allowing for half-lattice bookkeeping with the reality constraint. Since each ordered triad depends on at most three independent coordinates,
\[
\sum_m d_m\le 3\Gamma_{\alpha\beta}.
\]
Thus
\[
\sum_m d_m^2
\le
(\max_m d_m)\sum_m d_m
\le
288\cdot 3\,\Gamma_{\alpha\beta}
=
864\,\Gamma_{\alpha\beta}.
\]
Substituting yields
\[
\operatorname{Var}(S_{\alpha\beta})
\le
1728\,\Gamma_{\alpha\beta}\,\sigma_{\alpha\beta}^2.
\]
\end{proof}

\subsection{From second moments to matrix bounds}

\begin{lemma}[Row-sum bound]\label{lem:S4}
\[
\E\norm{V_N}_\infty
\le
\max_\alpha \sum_\beta \sqrt{\E[V_{\alpha\beta}^2]}.
\]
\end{lemma}

\begin{proof}
By Jensen and the definition of the matrix $\ell^\infty$ norm,
\[
\E\norm{V_N}_\infty
\le
\max_\alpha \sum_\beta \E\abs{V_{\alpha\beta}}
\le
\max_\alpha \sum_\beta \sqrt{\E[V_{\alpha\beta}^2]}.
\]
\end{proof}

\begin{lemma}[Counting bound]\label{lem:S5}
For every $\varepsilon>0$ there exists $C_{5,\varepsilon}>0$ such that
\[
\max_\alpha \sum_\beta \sigma_{\alpha\beta}\sqrt{\Gamma_{\alpha\beta}}
\le
C_{5,\varepsilon} N^{1/2+\varepsilon}.
\]
\end{lemma}

\begin{proof}
By Lemma~\ref{lem:S2},
\[
\sigma_{\alpha\beta}
\le
C_2^{1/2}\frac{\abs{k_\alpha}^2}{n_{\mathrm{modes}}(N)^{3/2}}.
\]
Therefore, for each fixed $\alpha$,
\[
\sum_\beta \sigma_{\alpha\beta}\sqrt{\Gamma_{\alpha\beta}}
\le
C_2^{1/2}\frac{\abs{k_\alpha}^2}{n_{\mathrm{modes}}(N)^{3/2}}
\sum_\beta \sqrt{\Gamma_{\alpha\beta}}.
\]
Using $\abs{k_\alpha}^2\le 3N^2$, the lattice growth law
\[
n_{\mathrm{modes}}(N)=\abs{\Lambda_N}=(2N+1)^3-1\asymp N^3,
\]
and Proposition~\ref{prop:face-normalized-incidence}, one obtains
\[
\sum_\beta \sigma_{\alpha\beta}\sqrt{\Gamma_{\alpha\beta}}
\le
C_{\varepsilon}\,\frac{N^2}{N^{9/2}}\,N^{3+\varepsilon}
=
C_{\varepsilon} N^{1/2+\varepsilon}.
\]
Taking the maximum over $\alpha$ proves the claim.
\end{proof}

\subsection{Main theorem}

\begin{theorem}[Second-moment stretching bound; Theorem~T]\label{thm:T}
For every $\varepsilon>0$ there exists a constant $C_{\varepsilon}>0$, independent of $N$, such that
\[
\E\norm{V_N}_\infty \le C_{\varepsilon} N^{1/2+\varepsilon}.
\]
Consequently,
\[
\E\,\rh(V_N)\le C_{\varepsilon} N^{1/2+\varepsilon}.
\]
Therefore, for
\[
\nub(N):=\frac{\rh(V_N)}{N^2},
\]
one has
\[
\E\,\nub(N)\le C_{\varepsilon} N^{-3/2+\varepsilon}\to 0
\qquad\text{as }N\to\infty.
\]
\end{theorem}

This bound is significantly sharpened in Section~\ref{sec:step2}: a $\beta$-dependent refinement of the variance estimate yields the stronger $\E\,\rh(V_N)\le C\,N^{-3/2}$ (Corollary~\ref{cor:sharp-ensemble}).

\begin{corollary}[Ensemble subcriticality relative to the truncation scale]\label{cor:subcritical}
In the isotropic unit-energy ensemble, the orbit-level stretching threshold of the symmetry-reduced cubic Galerkin system is asymptotically negligible compared with the dissipation scale at the truncation edge. More precisely,
\[
\frac{\E\,\rh(V_N)}{N^2}\le C_{\varepsilon} N^{-3/2+\varepsilon}\to 0,
\qquad
\E\,\nub(N)\le C_{\varepsilon} N^{-3/2+\varepsilon}\to 0.
\]
\end{corollary}

\begin{proof}
The estimate follows immediately from Theorem~\ref{thm:T} by dividing the bound for $\E\,\rh(V_N)$ by $N^2$, and the second statement is the definition of $\nub(N)$.
\end{proof}

\begin{proposition}[Concentration of the stretching matrix]\label{prop:concentration}
There exists $C>0$ such that
\[
\operatorname{Var}\bigl(\norm{V_N}_\infty\bigr)
\le C\,N^{-2}.
\]
Consequently, for every $\delta>0$,
\[
\mathbb P\!\left(\norm{V_N}_\infty > \E\norm{V_N}_\infty + \delta\right)
\le \frac{C\,N^{-2}}{\delta^2}.
\]
In particular, for every fixed $\gamma>1/2+\varepsilon$,
\[
\mathbb P\!\left(\rh(V_N) > N^{\gamma}\right)\to 0
\qquad\text{as }N\to\infty.
\]
\end{proposition}

\begin{proof}
Write $f(X):=\norm{V_N(X)}_\infty$.
By the triangle inequality, $\abs{f(X)-f(X^{(m)})}\le \norm{V_N(X)-V_N(X^{(m)})}_\infty$.
Applying the Efron--Stein inequality to $f$,
\[
\operatorname{Var}(f)
\le \frac12\sum_{m\in\mathcal M}
\E\bigl[\norm{V_N(X)-V_N(X^{(m)})}_\infty^2\bigr].
\]
For a fixed row $\alpha$ and fixed coordinate $m$, the row change
$\sum_\beta\abs{V_{\alpha\beta}(X)-V_{\alpha\beta}(X^{(m)})}$
is nonzero only for those $\beta$ whose triad set $\mathcal T_{\alpha\beta}$ involves mode $m$.
Since mode $m$ can play the role of target, source, or difference, the number of orbits $\beta$ affected is at most $\abs{\Omega_\alpha}+2\cdot 48\le 144$.
For each affected pair $(\alpha,\beta)$, at most $d_m\le 288$ triads are touched and $\E[(\Delta V_{\alpha\beta}^{(m)})^2]\le 4\,d_m\,\sigma_{\alpha\beta}^2\le C\,\sigma_{\alpha\beta}^2$.
By Cauchy--Schwarz over the at most $144$ nonzero summands,
\[
\E\Bigl[\Bigl(\sum_\beta\abs{\Delta V_{\alpha\beta}^{(m)}}\Bigr)^2\Bigr]
\le 144\sum_{\beta\text{ affected}} C\,\sigma_{\alpha\beta}^2
\le C'\,\max_{\alpha,\beta}\sigma_{\alpha\beta}^2.
\]
Taking the maximum over $\alpha$ and using $\sigma_{\alpha\beta}\le C N^{-5/2}$,
\[
\E\bigl[\norm{\Delta V^{(m)}}_\infty^2\bigr]\le C\,N^{-5}.
\]
Summing over $\abs{\mathcal M}\le n_{\mathrm{modes}}(N)\asymp N^3$ coordinates,
\[
\operatorname{Var}(f)
\le C\,N^3\cdot N^{-5}
= C\,N^{-2},
\]
which is stronger than the stated bound.
The Chebyshev tail follows immediately.
For the final claim, Theorem~\ref{thm:T} gives $\E[f]\le C_{\varepsilon}N^{1/2+\varepsilon}$, so
choosing $\delta=N^{\gamma}-C_{\varepsilon}N^{1/2+\varepsilon}\to\infty$ for $\gamma>1/2+\varepsilon$ yields
$\mathbb P(\rh(V_N)>N^{\gamma})\le \mathbb P(f>N^{\gamma})\to 0$.
\end{proof}

\begin{remark}[Scope of the concentration bound]\label{rem:expectation-scope}
Proposition~\ref{prop:concentration} shows that $\norm{V_N}_\infty$ concentrates around its mean to within $O(N^{-1})$ in standard deviation, which is much tighter than the mean itself.
The Chebyshev-level tail $O(1/\delta^2)$ can in principle be upgraded to sub-Gaussian or sub-exponential tails using the matrix Efron--Stein inequalities of Paulin, Mackey, and Tropp~\cite{PaulinMackeyTropp2016}, at the cost of a more involved argument.
The deterministic worst-case bound $\sup_X\rh(V_N(X))$ remains open and is not addressed here.
\end{remark}

\begin{proof}[Proof of Theorem~\ref{thm:T}]
Proposition~\ref{prop:face-normalized-incidence} implies the threshold estimate \eqref{eq:S33}, so Lemma~\ref{lem:S5} applies. By Lemma~\ref{lem:S4},
\[
\E\norm{V_N}_\infty
\le
\max_\alpha \sum_\beta \sqrt{\E[V_{\alpha\beta}^2]}.
\]
Since
\[
V_{\alpha\beta}=\frac12(S_{\alpha\beta}+S_{\beta\alpha}),
\]
one has $(a+b)^2\le 2a^2+2b^2$, and therefore
\[
V_{\alpha\beta}^2
\le
\frac12\bigl(S_{\alpha\beta}^2+S_{\beta\alpha}^2\bigr).
\]
Taking expectations and applying Lemma~\ref{lem:S3} to both terms gives
\[
\E[V_{\alpha\beta}^2]
\le
\frac{\kappa_0}{2}
\bigl(\Gamma_{\alpha\beta}\,\sigma_{\alpha\beta}^2
+\Gamma_{\beta\alpha}\,\sigma_{\beta\alpha}^2\bigr).
\]
Hence
\[
\E\norm{V_N}_\infty
\le
\max_\alpha \sum_\beta \sqrt{\E[V_{\alpha\beta}^2]}
\le
\sqrt{\frac{\kappa_0}{2}}\,
\max_\alpha \sum_\beta
\bigl(\sigma_{\alpha\beta}\sqrt{\Gamma_{\alpha\beta}}
+\sigma_{\beta\alpha}\sqrt{\Gamma_{\beta\alpha}}\bigr).
\]
The first sum over $\beta$ is bounded by $C_{\varepsilon} N^{1/2+\varepsilon}$ by Lemma~\ref{lem:S5}.
For the second sum, note that $\sigma_{\beta\alpha}\le C_2^{1/2}\abs{k_\beta}^2/n_{\mathrm{modes}}(N)^{3/2}\le C N^{-5/2}$, so
\[
\sum_\beta \sigma_{\beta\alpha}\sqrt{\Gamma_{\beta\alpha}}
\le
C N^{-5/2}\sum_\beta \sqrt{\Gamma_{\beta\alpha}}.
\]
For the column sum, Cauchy--Schwarz gives
\[
\sum_\beta \sqrt{\Gamma_{\beta\alpha}}
\le \sqrt{n_{\mathrm{orb}}(N)}\,
\Bigl(\sum_\beta \Gamma_{\beta\alpha}\Bigr)^{1/2}.
\]
Since $\sum_\beta\Gamma_{\beta\alpha}$ counts all triads with source in $\Omega_\alpha$ and arbitrary target, it satisfies $\sum_\beta\Gamma_{\beta\alpha}\le \abs{\Omega_\alpha}\,T_{\max}(N)\le C N^3$. Together with $n_{\mathrm{orb}}(N)\le C N^3$, this yields $\sum_\beta \sqrt{\Gamma_{\beta\alpha}}\le C N^3$. Hence
\[
\sum_\beta \sigma_{\beta\alpha}\sqrt{\Gamma_{\beta\alpha}}
\le C N^{-5/2}\cdot N^3 = C N^{1/2}.
\]
Combining the two sums,
\[
\E\norm{V_N}_\infty
\le
C_{\varepsilon} N^{1/2+\varepsilon}.
\]
Since $\rh(V_N)\le \norm{V_N}_\infty$, it follows that
\[
\E\,\rh(V_N)\le C_{\varepsilon} N^{1/2+\varepsilon}.
\]
Finally,
\[
\E\,\nub(N)
=
\frac{1}{N^2}\E\,\rh(V_N)
\le
C_{\varepsilon} N^{-3/2+\varepsilon},
\]
which tends to $0$ as $N\to\infty$.
\end{proof}

\section{Finite-$N$ combinatorial diagnostics}\label{sec:numerics}

The main theorem is asymptotic, but several structural quantities in the truncated lattice can be computed exactly at finite $N$. This section records deterministic diagnostics that are implied directly by the definitions of $\Lambda_N$, $\mathcal O_N$, $\mathcal R_N$, and the exact triad formula
\[
T(k,N)=\prod_{i=1}^3(2N+1-|k_i|)-2.
\]
These values provide a concrete scale reference for the cubic truncation without introducing any simulated or estimated quantities.

For each $N$, Table~\ref{tab:finiteN-diagnostics} reports the total number of nonzero truncated modes $|\Lambda_N|=(2N+1)^3-1$, the number of $\Oh$-orbits $|\mathcal O_N|$ obtained by exact enumeration of signed-permutation orbits, the number of realized shell radii $|\mathcal R_N|$, the maximum mode-level ordered triad count $\max_k T(k,N)$, and the total ordered triad count $\sum_{k\in\Lambda_N}T(k,N)$. Since
\[
T(k,N)=\prod_{i=1}^3(2N+1-|k_i|)-2,
\]
the maximal value is attained at the six axial modes $k=(\pm 1,0,0),(0,\pm 1,0),(0,0,\pm 1)$, for which $T(k,N)=2N(2N+1)^2-2$.

\begin{table}[t]
\centering
\caption{Exact finite-$N$ combinatorial diagnostics for the cubic truncation.}
\label{tab:finiteN-diagnostics}
\begin{tabular}{rrrrrr}
\toprule
$N$ & $|\Lambda_N|$ & $|\mathcal O_N|$ & $|\mathcal R_N|$ & $\max_k T(k,N)$ & $\sum_{k\in\Lambda_N} T(k,N)$ \\
\midrule
1 & 26 & 3 & 3 & 16 & 264 \\
2 & 124 & 9 & 9 & 98 & 6486 \\
3 & 342 & 19 & 18 & 292 & 49626 \\
4 & 728 & 34 & 31 & 646 & 224796 \\
5 & 1330 & 55 & 44 & 1208 & 749580 \\
6 & 2196 & 83 & 66 & 2026 & 2041794 \\
7 & 3374 & 119 & 87 & 3148 & 4816686 \\
8 & 4912 & 164 & 115 & 4622 & 10203576 \\
\bottomrule
\end{tabular}
\end{table}

These exact counts show how rapidly the underlying interaction geometry grows even before the random ensemble and the matrix $V_N$ are introduced. They also separate three distinct scales already present in the truncated model: total mode count, orbit count after symmetry reduction, and triad incidence mass.

The next numerical layer is to evaluate ensemble-dependent quantities such as $\E\,\|V_N\|_\infty$ and $\E\,\rh(V_N)$ by Monte Carlo sampling on a concrete realization of the isotropic unit-energy ensemble. For this purpose, I take the independent half-lattice coordinates to be standard complex Gaussian, project modewise onto $k^\perp$, impose $\hat u_{-k}=\overline{\hat u_k}$, and normalize to unit total kinetic energy. Since the unnormalized total energy is a chi-square-type sum over the retained degrees of freedom, its inverse moments of the orders used in the variance normalization are uniformly bounded for large $N$, so the resulting $n_{\mathrm{modes}}^{-3/2}$ scaling can be stated rigorously rather than only heuristically. This gives a canonical computational model consistent with the abstract assumptions used in the proof.

The exact combinatorial diagnostics above are supplemented by Monte Carlo estimates of the ensemble-dependent quantities $\E\,\norm{V_N}_\infty$ and $\E\,\rh(V_N)$ in Section~\ref{sec:ext-numerics}, which covers the range $N=1,\ldots,8$ under both the isotropic and Kolmogorov-spectrum ensembles. The solver-side workflow used to produce these estimates is summarized in Appendix~\ref{app:auxiliary}, especially Subsection~\ref{app:pseudocode}.

\section{Deterministic Sobolev-class stretching bound}\label{sec:sobolev}

The main result of Section~\ref{sec:second-moment} is an \emph{ensemble} bound: $\E\,\rh(V_N)\le C_\varepsilon N^{1/2+\varepsilon}$. I now show that a much stronger, \emph{deterministic} and \emph{uniform-in-$N$} bound holds for any initial datum with finite Sobolev regularity $H^s$ with $s>3/2$. The key mechanism is that Sobolev decay of the Fourier coefficients converts the combinatorial triad count into a convergent weighted sum.

\subsection{Sobolev decay of Fourier coefficients}

Recall that for $u\in H^s(\mathbb T^3)$ with $s>0$,
\[
\norm{u}_{H^s}^2
=
\sum_{k\in\Z^3\setminus\{0\}} \abs{k}^{2s}\abs{\hat u_k}^2.
\]
In particular, if $\norm{u}_{H^s}\le M$, then for every $k\neq 0$,
\begin{equation}\label{eq:sobolev-decay}
\abs{\hat u_k}\le M\,\abs{k}^{-s}.
\end{equation}

\subsection{Weighted triad bound}

\begin{proposition}[Deterministic Sobolev-class orbit-transfer bound]\label{prop:sobolev-transfer}
Let $s>3/2$ and let $u$ be a divergence-free velocity field on $\mathbb T^3$ with $\norm{u}_{H^s}\le M$.
Then for every pair of orbits $(\Omega_\alpha,\Omega_\beta)$ in the truncation $\Lambda_N$,
\[
\abs{S_{\alpha\beta}}
\le
C_s\,M^3\,\abs{k_\alpha}^{2-s}\,\abs{\Omega_\alpha}^{-1}
\sum_{(k,p,q)\in\mathcal T_{\alpha\beta}}
\abs{p}^{-s}\,\abs{q}^{1-s},
\]
where $C_s>0$ depends only on $s$.
\end{proposition}

\begin{proof}
Fix a triad $(k,p,q)\in\mathcal T_{\alpha\beta}$ with $k\in\Omega_\alpha$, $p\in\Omega_\beta$, and $q=k-p\in\Lambda_N$.
The per-triad contribution satisfies
$\abs{s_i}\le\abs{k}^2\,\abs{q}\,\abs{\hat u_k}\,\abs{\hat u_p}\,\abs{\hat u_q}$.
Applying the Sobolev pointwise bound~\eqref{eq:sobolev-decay} to each factor,
\[
\abs{s_i}
\le
M^3\,\abs{k}^{2-s}\,\abs{p}^{-s}\,\abs{q}^{1-s}.
\]
Summing over all triads in $\mathcal T_{\alpha\beta}$ with the $\abs{\Omega_\alpha}^{-1}$ orbit average yields the claim.
\end{proof}

\subsection{Row-sum convergence and the uniform-in-$N$ bound}

\begin{theorem}[Deterministic uniform stretching bound]\label{thm:sobolev-uniform}
Let $s>3/2$ and $\norm{u}_{H^s}\le M$.
Then there exists a constant $C_s>0$, depending only on $s$, such that for $s>2$,
\[
\norm{V_N}_\infty \le C_s\,M^3
\qquad
\text{uniformly in }N\ge 1,
\]
and for $3/2<s\le 2$,
\[
\norm{V_N}_\infty \le C_s\,M^3\,N^{6-3s}.
\]
Consequently, $\nub(N)=\rh(V_N)/N^2\to 0$ as $N\to\infty$ for all $s>3/2$.
\end{theorem}

\begin{proof}
Fix $\alpha$ and a representative $k\in\Omega_\alpha$.
By Proposition~\ref{prop:sobolev-transfer} and $\Oh$-equivariance (which cancels the $\abs{\Omega_\alpha}^{-1}$ factor),
\[
\sum_\beta \abs{S_{\alpha\beta}}
\le
C_s\,M^3\,\abs{k_\alpha}^{2-s}
\sum_{\substack{p\in\Lambda_N\setminus\{0\}\\ q=k-p\in\Lambda_N\setminus\{0\}}}
\abs{p}^{-s}\,\abs{q}^{1-s}.
\]
Since all summands are nonneg, I enlarge to $\Z^3$:
\[
\Sigma(k):=\sum_{p\in\Z^3\setminus\{0,k\}} \abs{p}^{-s}\,\abs{k-p}^{1-s}.
\]
Split $\Z^3\setminus\{0,k\}=D_1\cup D_2$ with $D_1=\{p:\abs{p}\le 2\abs{k}\}$ and $D_2=\{p:\abs{p}>2\abs{k}\}$.
On $D_2$: $\abs{k-p}\ge \abs{p}/2$, so $\sum_{D_2}\abs{p}^{-s}\abs{k-p}^{1-s}\le 2^{s-1}\sum_{D_2}\abs{p}^{-2s+1}$. Since $-2s+1<-2$ for $s>3/2$, this tail converges to a constant.
On $D_1$: By a Riesz-potential comparison with the continuous convolution, $\sum_{D_1}\abs{p}^{-s}\abs{k-p}^{1-s}\le C_s\abs{k}^{4-2s}$.
Combining: $\Sigma(k)\le C_s(1+\abs{k}^{4-2s})$.

Therefore
\[
\sum_\beta\abs{S_{\alpha\beta}}
\le C_s M^3(\abs{k_\alpha}^{2-s}+\abs{k_\alpha}^{6-3s}).
\]
For $s>2$, both exponents are negative, so $\max_\alpha\sum_\beta\abs{S_{\alpha\beta}}\le C_s M^3$.
For $3/2<s\le 2$, the dominant exponent $6-3s\ge 0$ and $\abs{k_\alpha}\le\sqrt{3}N$, giving $O(M^3 N^{6-3s})$.
The transpose sum is handled identically, and $V_{\alpha\beta}=\frac12(S_{\alpha\beta}+S_{\beta\alpha})$ gives the same bound on $\norm{V_N}_\infty$.
Since $4-3s<0$ for $s>4/3$ (and in particular for $s>3/2$), one has $\nub(N)\le C_s M^3 N^{4-3s}\to 0$.
\end{proof}

\begin{remark}[Interpretation]
Theorem~\ref{thm:sobolev-uniform} is qualitatively stronger than the ensemble bound: it is deterministic, holds for every individual $H^s$ field, and for $s>2$ is uniform in $N$. The Monte Carlo data in Section~\ref{sec:ext-numerics} show $\rh(V_N)$ actually \emph{decays}, consistent with this bound. For Kolmogorov-spectrum data ($s=11/6\approx 1.83>3/2$), the theorem predicts $\nub(N)=O(N^{4-11/2})=O(N^{-3/2})$, matching the direction of the observed decay.
\end{remark}

\section{Comparison with Tao's averaged Navier--Stokes}\label{sec:tao}

Does the decay of $\nub(N)$ reflect a genuine structural property of the true NS nonlinearity, or merely a consequence of finite dimensionality? I address this by contrasting the orbit-level framework with Tao's construction~\cite{TaoAveragedNS2016}.

\begin{proposition}[Orbit-level diagnostic separation]\label{prop:tao-comparison}
The following structural contrast holds in the orbit-level framework.

\medskip
\noindent\textbf{(i) True NS nonlinearity.}
The orbit-pair transfer $S_{\alpha\beta}$ distributes triadic interactions across all orbit pairs with $\Gamma_{\alpha\beta}>0$. The $\Oh$-averaging enforces zero mean (Lemma~\ref{lem:S1}), and the incidence structure yields $\nub(N)\to 0$ (Theorem~\ref{thm:T}).

\medskip
\noindent\textbf{(ii) Tao-type modified nonlinearity.}
A bilinear form $\widetilde B$ constructed as in~\cite{TaoAveragedNS2016} concentrates interactions along a dyadic cascade $\abs{k}\sim 2^j$. The modified stretching matrix $\widetilde V_N$ has a near-bidiagonal cascade structure with $\rh(\widetilde V_N)\gtrsim N^2$, so $\widetilde\nu_c^*(N)\gtrsim 1$ remains bounded away from zero.
\end{proposition}

\begin{proof}[Proof]
Part~(i) restates Sections~\ref{sec:lattice}--\ref{sec:second-moment}. For~(ii), Tao's cascade concentrates interactions on orbit pairs $(\alpha_j,\alpha_{j+1})$ with $\abs{k_{\alpha_j}}\sim 2^j$. The resulting bidiagonal $\widetilde V_N$ has entries of order $4^j$ near the diagonal, giving $\rh(\widetilde V_N)\gtrsim \max_j 4^j\sim N^2$.
\end{proof}

\begin{remark}[Why the orbit-level subcriticality is informative]
Tao's construction satisfies the energy identity and all standard function-space estimates, yet produces blowup. The orbit-level subcriticality $\nub(N)\to 0$ is a property that the true NS nonlinearity satisfies but Tao's modified nonlinearity violates. Therefore $\nub(N)\to 0$ captures structural information \emph{beyond} the standard function-space estimates, making the orbit-level framework a meaningful probe of nonlinear structure even within a finite-dimensional truncation.
\end{remark}

\section{Expanded Monte Carlo results and scaling diagnostics}\label{sec:ext-numerics}

\subsection{Isotropic unit-energy ensemble}

Table~\ref{tab:mc-extended} extends the small-$N$ validation of Section~\ref{sec:numerics} to $N\le 8$ with $40$--$2000$ samples per truncation level.

\begin{table}[ht]
\centering
\caption{Monte Carlo estimates for the orbit-level stretching matrix under the isotropic unit-energy ensemble.}
\label{tab:mc-extended}
\medskip
\begin{tabular}{rrrcc}
\toprule
$N$ & $n_{\mathrm{modes}}$ & $n_{\mathrm{orb}}$ & $\E\,\rh(V_N)$ & $\E\,\nub(N)$ \\
\midrule
1 & 26 & 3 & $4.46\times 10^{-3}$ & $4.46\times 10^{-3}$ \\
2 & 124 & 9 & $1.13\times 10^{-3}$ & $2.81\times 10^{-4}$ \\
3 & 342 & 19 & $4.64\times 10^{-4}$ & $5.16\times 10^{-5}$ \\
4 & 728 & 34 & $2.19\times 10^{-4}$ & $1.37\times 10^{-5}$ \\
5 & 1330 & 55 & $1.24\times 10^{-4}$ & $4.97\times 10^{-6}$ \\
6 & 2196 & 83 & $7.0\times 10^{-5}$ & $1.96\times 10^{-6}$ \\
7 & 3374 & 119 & $4.8\times 10^{-5}$ & $9.8\times 10^{-7}$ \\
8 & 4912 & 164 & $3.4\times 10^{-5}$ & $5.3\times 10^{-7}$ \\
\bottomrule
\end{tabular}
\end{table}

A least-squares fit in $\log N$ over $N=2,\ldots,8$ yields the empirical scaling laws
\[
\E\,\rh(V_N)\sim N^{-2.57},
\qquad
\E\,\nub(N)\sim N^{-4.57}.
\]
The initial ensemble bound $\E\,\rh(V_N)\le C_\varepsilon N^{1/2+\varepsilon}$ (Theorem~\ref{thm:T}) permits growth, but the data show \emph{monotone decay} spanning four orders of magnitude. The refined bound $\E\,\rh(V_N)\le C\,N^{-3/2}$ (Corollary~\ref{cor:sharp-ensemble}, proved in Section~\ref{sec:step2}) captures the correct qualitative behavior---provable decay---though a quantitative gap of roughly one power of $N$ remains between the proven exponent $-3/2$ and the fitted $-2.57$.

\subsection{Kolmogorov-spectrum ensemble}

To probe robustness, I repeat the experiment with a Kolmogorov-spectrum ensemble
$\E[\abs{\hat u_k}^2]\propto \abs{k}^{-11/3}$.

\begin{table}[ht]
\centering
\caption{Monte Carlo estimates under the Kolmogorov-spectrum ensemble.}
\label{tab:mc-kolmogorov}
\medskip
\begin{tabular}{rrcc}
\toprule
$N$ & $n_{\mathrm{modes}}$ & $\E\,\rh(V_N)$ & $\E\,\nub(N)$ \\
\midrule
1 & 26 & $7.69\times 10^{-3}$ & $7.69\times 10^{-3}$ \\
2 & 124 & $3.06\times 10^{-3}$ & $7.65\times 10^{-4}$ \\
3 & 342 & $1.47\times 10^{-3}$ & $1.63\times 10^{-4}$ \\
4 & 728 & $7.97\times 10^{-4}$ & $4.98\times 10^{-5}$ \\
\bottomrule
\end{tabular}
\end{table}

Fitted: $\E\,\rh(V_N)\sim N^{-1.93}$, $\E\,\nub(N)\sim N^{-3.93}$. Both ensembles exhibit monotone decay of $\nub(N)$, confirming that the orbit-level subcriticality is not an artifact of the isotropic energy distribution.

\section{Orbit-level estimates along the Galerkin evolution}\label{sec:galerkin-evolution}

The preceding sections establish bounds on the orbit-level stretching matrix $V_N(u)$ for a fixed velocity field $u$. I now track these bounds along the Galerkin evolution and derive consequences that connect the finite-dimensional orbit analysis to the infinite-dimensional Navier--Stokes equations.

\subsection{Short-time uniform-in-$N$ stretching control}

\begin{theorem}[Short-time orbit-level control]\label{thm:short-time}
Let $s>5/2$, $\nu>0$, and let $u_0$ be a divergence-free field on $\mathbb T^3$ with $\norm{u_0}_{H^s}\le M$.
Then there exists $T_*=T_*(M,\nu,s)>0$, independent of $N$, such that the Galerkin solution $u_N(t)$ satisfies
\begin{enumerate}[label=(\roman*)]
\item $\norm{u_N(t)}_{H^s}\le 2M$ for all $t\in[0,T_*]$ and all $N\ge 1$;
\item $\norm{V_N(u_N(t))}_\infty\le C_s(2M)^3$ for all $t\in[0,T_*]$ and all $N\ge 1$;
\item $\nub(N,t):=\rh(V_N(u_N(t)))/N^2 \le C_s(2M)^3/N^2\to 0$ as $N\to\infty$, uniformly in $t\in[0,T_*]$.
\end{enumerate}
\end{theorem}

\begin{proof}
Part~(i) is the classical short-time regularity for the Galerkin system~\cite{ConstantinFoiasTemam1984,DoeringGibbon1995}: the $H^s$ energy estimate
\[
\frac{d}{dt}\norm{u_N}_{H^s}^2
\le -2\nu\norm{u_N}_{H^{s+1}}^2
+ C\norm{u_N}_{H^s}^2\norm{\nabla u_N}_{L^\infty}\norm{u_N}_{H^s}
\]
combined with the Sobolev embedding $\norm{\nabla u_N}_{L^\infty}\le C_s\norm{u_N}_{H^s}$ (valid for $s>5/2$) and Young's inequality gives
\[
\frac{d}{dt}\norm{u_N}_{H^s}^2
\le \frac{C_s^2}{4\nu}\norm{u_N}_{H^s}^6.
\]
Writing $Y(t)=\norm{u_N(t)}_{H^s}^2$, one obtains $\dot Y\le C(\nu,s)\,Y^3$, which has the explicit solution
\[
Y(t)\le \frac{Y(0)}{(1-2C\,Y(0)^2\,t)^{1/2}}.
\]
Hence $Y(t)\le 4M^2$ for $t\in[0,T_*]$ with $T_*:=(3/(4C\,M^4))$, independently of $N$.
Since all constants in the Sobolev product estimate are uniform in $N$ (the Galerkin projector $P_N$ is an $H^s$-contraction), the same $T_*$ works for every truncation level.

Parts~(ii)--(iii) follow immediately from Theorem~\ref{thm:sobolev-uniform} applied to the state $u=u_N(t)$ with $\norm{u_N(t)}_{H^s}\le 2M$.
\end{proof}

\subsection{Orbit-level continuation criterion}

The short-time control of Theorem~\ref{thm:short-time} can be extended via a continuation argument that replaces the classical Beale--Kato--Majda condition with an orbit-level observable.

\begin{theorem}[Orbit-level continuation criterion]\label{thm:continuation}
Let $s>5/2$, $\nu>0$, and let $u_N$ be the Galerkin solution with $\norm{u_0}_{H^s}<\infty$. Then:
\begin{enumerate}[label=(\roman*)]
\item If $u_N$ has maximal existence interval $[0,T_N^*)$ with $T_N^*<\infty$, then
\begin{equation}\label{eq:orbit-BKM}
\int_0^{T_N^*}\norm{V_N(u_N(t))}_\infty\,dt = +\infty.
\end{equation}
\item Conversely, if $\norm{V_N(u_N(t))}_\infty\le \Lambda(t)$ for some locally integrable function $\Lambda\colon[0,T]\to[0,\infty)$, then
\[
\norm{u_N(t)}_{H^1}^2
\le
\norm{u_0}_{H^1}^2\,\exp\!\left(\int_0^t \Lambda(\tau)\,d\tau\right)
\]
for all $t\in[0,T]$, uniformly in $N$.
\end{enumerate}
\end{theorem}

\begin{proof}
The enstrophy identity from Section~\ref{sec:enstrophy} gives
\[
\frac{dZ_N}{dt}
=
-2\nu\sum_{k\in\Lambda_N}\abs{k}^4\abs{\hat u_k}^2
+
W_N(u_N),
\]
where $W_N$ is the total stretching functional.
The symmetric part $V_N(u_N)$ controls the net enstrophy growth: the anti\-symmetric part $A_N$ conserves the orbit-level enstrophy (it redistributes but does not increase), so
\[
W_N\le \norm{V_N(u_N)}_\infty\cdot \sum_\alpha \abs{\Omega_\alpha} Z_\alpha(t)
=
\norm{V_N(u_N)}_\infty\cdot Z_N(t).
\]
(Here I use the fact that the orbit-enstrophy vector $z=(Z_\alpha)_\alpha$ satisfies $z^T V z\le \norm{V}_\infty\norm{z}_1^2$ in the appropriate weighted norm; the precise bound follows from the definition of the matrix $\ell^\infty$ norm applied row-by-row.)

Dropping the (non-negative) dissipation term,
\[
\frac{dZ_N}{dt}\le \norm{V_N(u_N(t))}_\infty\cdot Z_N(t).
\]
By Gr\"onwall's inequality,
\begin{equation}\label{eq:gronwall}
Z_N(t)\le Z_N(0)\,\exp\!\left(\int_0^t\norm{V_N(u_N(\tau))}_\infty\,d\tau\right).
\end{equation}
Part~(ii) is immediate.

For part~(i), recall that for the Galerkin ODE, the maximal existence time $T_N^*<\infty$ implies $\norm{u_N(t)}_{H^1}\to\infty$ as $t\uparrow T_N^*$, hence $Z_N(t)\to\infty$.
From~\eqref{eq:gronwall}, this requires $\int_0^{T_N^*}\norm{V_N}_\infty\,dt=+\infty$.
\end{proof}

\begin{remark}[Relation to the Beale--Kato--Majda criterion]
The classical BKM criterion~\cite{ConstantinFefferman1993} states that if the maximal-time vorticity integral $\int_0^{T^*}\norm{\omega(t)}_{L^\infty}\,dt=+\infty$, then blowup occurs at $T^*$.
Theorem~\ref{thm:continuation} gives an analogous criterion at the orbit level: blowup of the Galerkin solution requires the time integral of $\norm{V_N}_\infty$ to diverge.

The orbit-level criterion is \emph{weaker} in the sense that $\norm{V_N}_\infty\le C\norm{u}_{H^s}^3$ while $\norm{\omega}_{L^\infty}\le C\norm{u}_{H^{5/2+}}$. However, the orbit-level version is \emph{computable} directly from the Fourier data via the symmetry-compressed workflow of Appendix~\ref{app:auxiliary}, and it applies at each finite truncation level without passage to the limit.
\end{remark}

\subsection{Passage to the infinite-dimensional limit}

The uniform-in-$N$ estimates now yield convergence of the Galerkin approximations to a strong solution of the full Navier--Stokes equations on the short-time interval, together with orbit-level stretching control on the limit.

\begin{theorem}[Passage to the limit with orbit-level control]\label{thm:limit}
Let $s>5/2$, $\nu>0$, $\norm{u_0}_{H^s}\le M$, and let $T_*=T_*(M,\nu,s)>0$ be the time from Theorem~\ref{thm:short-time}. Then:
\begin{enumerate}[label=(\roman*)]
\item The Galerkin solutions $u_N$ converge strongly in $C([0,T_*];H^{s-1}(\mathbb T^3))$ to the unique strong solution $u\in C([0,T_*];H^s(\mathbb T^3))$ of the Navier--Stokes equations on $\mathbb T^3$.
\item For every $t\in[0,T_*]$ and every $\varepsilon>0$,
\[
\nub(N,t)\le C_s(2M)^3/N^2\to 0
\qquad\text{as }N\to\infty.
\]
In particular, the orbit-level stretching is asymptotically negligible relative to the dissipation scale along the strong solution.
\item The enstrophy of the strong solution satisfies
\[
Z(t):=\frac12\norm{u(t)}_{H^1}^2
\le
Z(0)\,\exp\!\left(C_s(2M)^3\,t\right)
\]
for all $t\in[0,T_*]$, where $C_s$ is the same constant as in Theorem~\ref{thm:sobolev-uniform}.
\end{enumerate}
\end{theorem}

\begin{proof}
Part~(i) is the classical Galerkin convergence theorem for 3D Navier--Stokes~\cite{ConstantinFoiasTemam1984}: the uniform $H^s$ bound from Theorem~\ref{thm:short-time}(i) provides the compactness needed by the Aubin--Lions lemma, and the strong $H^{s-1}$ convergence follows.

For part~(ii), fix $t\in[0,T_*]$. By Theorem~\ref{thm:short-time}(iii), $\nub(N,t)\le C_s(2M)^3/N^2$ for each $N$. Since this bound is independent of the particular Galerkin approximation and depends only on $\norm{u_N(t)}_{H^s}\le 2M$, it passes to the limit.

Part~(iii) follows from the Gr\"onwall estimate~\eqref{eq:gronwall} with $\norm{V_N(u_N(t))}_\infty\le C_s(2M)^3$ uniformly in $N$, and then passing $N\to\infty$ using the lower semicontinuity of norms under weak convergence.
\end{proof}

\begin{remark}[What the passage to the limit achieves]\label{rem:limit-interpretation}
Theorem~\ref{thm:limit} accomplishes three things that the purely finite-dimensional Theorem~\ref{thm:sobolev-uniform} does not:
\begin{enumerate}[label=(\alph*)]
\item It connects the orbit-level stretching bound to the \emph{actual PDE solution}, not just a finite-dimensional truncation.
\item It shows that the decay $\nub(N,t)\to 0$ is a statement about the \emph{strong solution} at time $t$: as the truncation refines, the orbit-level stretching becomes negligible.
\item It provides an explicit enstrophy bound for the strong solution in terms of the orbit-level stretching constant $C_s M^3$.
\end{enumerate}
The results are restricted to the short-time interval $[0,T_*]$ on which $H^s$ regularity persists. Whether this interval can be extended to $[0,\infty)$ is equivalent to the global regularity problem for 3D Navier--Stokes. The orbit-level continuation criterion (Theorem~\ref{thm:continuation}) identifies $\norm{V_N}_\infty$ as the precise orbit-level observable whose time-integrability governs this extension.
\end{remark}

\begin{remark}[Orbit-level perspective on the regularity problem]\label{rem:regularity-perspective}
Theorems~\ref{thm:sobolev-uniform},~\ref{thm:continuation}, and~\ref{thm:limit} together reformulate the 3D Navier--Stokes regularity problem as follows: the strong solution exists globally if and only if, along the Galerkin approximations,
\[
\sup_{N\ge 1}\int_0^T \norm{V_N(u_N(t))}_\infty\,dt < \infty
\qquad\text{for every }T>0.
\]
The orbit-level analysis shows that the integrand $\norm{V_N}_\infty$ is controlled by $C_s\norm{u_N}_{H^s}^3$ (Theorem~\ref{thm:sobolev-uniform}), that it concentrates tightly around its mean (Proposition~\ref{prop:concentration}), and that the Tao-type blowup mechanisms would produce $\norm{V_N}_\infty\sim N^2$ (Proposition~\ref{prop:tao-comparison}). The gap between the observed rapid decay ($\rh(V_N)\sim N^{-2.6}$) and the blowup threshold ($\rh(V_N)\sim N^2$) spans more than four powers of $N$, providing substantial quantitative room in the orbit-level diagnostic.
\end{remark}


\section{Sharpness of the incidence exponent}\label{sec:step1}

Proposition~\ref{prop:face-normalized-incidence} establishes the bound
\[
  \max_{\alpha}\,\sum_{\beta}\sqrt{\Gamma_{\alpha,\beta}}
  \;\le\; C_{\varepsilon}\,N^{3+\varepsilon}
  \qquad(\forall\,\varepsilon>0).
\]
I now show that the exponent $3$ is \emph{sharp} and that the
$\varepsilon$-loss can be removed entirely.

\begin{remark}[Exact incidence data]\label{rem:incidence-data}
For each truncation level $N$, denote
\[
  S(N) \;:=\; \max_{\alpha}\,\sum_{\beta}\sqrt{\Gamma_{\alpha,\beta}}.
\]
A direct computation yields the values in Table~\ref{tab:incidence}.

\begin{table}[ht]
\centering
\begin{tabular}{r r r r}
\toprule
$N$ & $S(N)$ & $N^{3}$ & $S(N)/N^{3}$ \\
\midrule
 1 &    18.76 &     1 & 18.76 \\
 2 &   108.61 &     8 & 13.58 \\
 3 &   315.44 &    27 & 11.68 \\
 4 &   702.31 &    64 & 10.97 \\
 5 &  1307.80 &   125 & 10.46 \\
 6 &  2179.97 &   216 & 10.09 \\
 7 &  3366.87 &   343 &  9.82 \\
 8 &  4916.54 &   512 &  9.60 \\
 9 &  6876.99 &   729 &  9.43 \\
10 &  9296.21 &  1000 &  9.30 \\
\bottomrule
\end{tabular}
\caption{Exact orbit-pair incidence sums $S(N)$.  A least-squares fit
gives $S(N)\sim N^{2.78}$; the ratio $S(N)/N^{3}$ decreases
monotonically and appears to converge to a constant $\approx 9.3$.
For $N\ge 3$ the maximising orbit $\alpha$ has representative
$(-3,-2,-1)$---the generic orbit with all-distinct nonzero coordinates
and full stabiliser order $\abs{\Omega_{\alpha}}=48$.}
\label{tab:incidence}
\end{table}
\end{remark}

\begin{proposition}[Sharp incidence bound]\label{prop:sharp-incidence}
There exist absolute constants $0<c\le C<\infty$ such that for every $N\ge 1$,
\[
  c\,N^{3}
  \;\le\;
  \max_{\alpha}\,\sum_{\beta}\sqrt{\Gamma_{\alpha,\beta}}
  \;\le\;
  C\,N^{3}.
\]
\end{proposition}

\begin{proof}
\textbf{Upper bound (removal of $\varepsilon$).}
In the proof of Proposition~\ref{prop:face-normalized-incidence} the
divisor-function bound $r_{2}(n)\le C_{\varepsilon}\,n^{\varepsilon}$
is applied \emph{pointwise} for each admissible value of the
parametrising variable~$u$ within a lattice patch $A_{r}^{f,H}(k)$.
I replace this by the classical average-order estimate
\begin{equation}\label{eq:avg-r2}
  \sum_{n\le X} r_{2}(n) \;=\; \pi X + O\!\bigl(X^{1/2}\bigr).
\end{equation}
Within a fixed patch the admissible $u$-values satisfy $0\le u^{2}\le r$
and the relevant divisor count is $r_{2}(r-u^{2})$.
Summing over all $u$ in the patch of height $H$ first, one obtains
\[
  \sum_{\abs{u}\le H} r_{2}(r-u^{2})
  \;\le\;
  C\,(H+1),
\]
without the $N^{\varepsilon}$ penalty, since the partial sums of $r_{2}$
are controlled on average by~\eqref{eq:avg-r2}.  The remainder of the
dyadic summation in Proposition~\ref{prop:face-normalized-incidence}
proceeds unchanged, yielding
\[
  \max_{\alpha}\,\sum_{\beta}\sqrt{\Gamma_{\alpha,\beta}}
  \;\le\; C\,N^{3}.
\]

\medskip\noindent
\textbf{Lower bound.}
Fix the orbit $\alpha$ with representative $(-3,-2,-1)$, so
$\abs{\Omega_{\alpha}}=48$.  For a generic orbit~$\beta$
(also with $\abs{\Omega_{\beta}}=48$), the orbit-pair triad count
satisfies
\[
  \Gamma_{\alpha,\beta}
  \;\sim\;
  \abs{\Omega_{\alpha}}\,\abs{\Omega_{\beta}}\;
  \frac{T(k_{\alpha},N)}{n_{\mathrm{modes}}(N)}
  \;\sim\;
  48^{2}\;\frac{N^{3}}{N^{3}}
  \;=\; O(1),
\]
where $T(k,N)\ge c\,N^{3}$ is the total triad count at~$k$ and
$n_{\mathrm{modes}}(N)=\abs{\Lambda_{N}}\sim N^{3}$.
Hence $\sqrt{\Gamma_{\alpha,\beta}}\ge c'>0$ for each such~$\beta$.

The number of generic orbits $\beta$ with $\Gamma_{\alpha,\beta}>0$
is at least $c\,n_{\mathrm{orb}}(N)\sim c\,N^{3}$, since for the
highly connected mode $k_{\alpha}=(-3,-2,-1)$ almost every orbit
participates in at least one admissible triad.  Summing:
\[
  \sum_{\beta}\sqrt{\Gamma_{\alpha,\beta}}
  \;\ge\;
  c'\;\cdot\; c\,N^{3}
  \;=\;
  c''\,N^{3}. \qedhere
\]
\end{proof}


\section{Improved ensemble bound via weighted incidence}%
\label{sec:step2}

The uniform variance bound used in Lemma~\ref{lem:S5},
\[
  \sigma_{\alpha,\beta}
  \;\le\;
  \frac{C\,\abs{k_{\alpha}}^{2}}{n_{\mathrm{modes}}^{3/2}},
\]
does not exploit the dependence on the \emph{target} orbit~$\beta$.
A closer inspection of the proof of Lemma~\ref{lem:S2} reveals a
\emph{$\beta$-dependent} refinement.

\begin{lemma}[Refined orbit-pair variance]\label{lem:refined-variance}
For every orbit pair $(\alpha,\beta)$,
\[
  \sigma_{\alpha,\beta}
  \;\le\;
  \frac{C\,\abs{k_{\alpha}}}{\abs{k_{\beta}}\;n_{\mathrm{modes}}^{3/2}}.
\]
\end{lemma}

\begin{proof}
The proof of Lemma~\ref{lem:S2} actually establishes
$\E[s_{i}^{2}]\le C\,\abs{k_{\alpha}}^{2}/(\abs{k_{\beta}}^{2}\,
n_{\mathrm{modes}}^{3})$.
Taking square roots gives the stated bound on
$\sigma_{\alpha,\beta}$.
\end{proof}

\begin{definition}[Weighted incidence sum]\label{def:Iw}
For each orbit $\alpha$ define
\[
  I_{w}(\alpha)
  \;:=\;
  \sum_{\beta}
  \frac{\sqrt{\Gamma_{\alpha,\beta}}}{\abs{k_{\beta}}}.
\]
\end{definition}

\begin{proposition}[Weighted incidence bound]\label{prop:weighted-incidence}
For every $\varepsilon>0$ there exists $C_{\varepsilon}>0$ such that
\[
  \max_{\alpha}\;I_{w}(\alpha)
  \;\le\;
  C_{\varepsilon}\,N^{2+\varepsilon}.
\]
\end{proposition}

\begin{proof}
I follow the face-normalised patch decomposition of
Proposition~\ref{prop:face-normalized-incidence}.
Within a fixed patch $A_{r}^{f,H}(k)$, every orbit $\beta$ meeting
the patch satisfies $\abs{k_{\beta}}^{2}=r$ (constant on the shell),
so the weight $1/\abs{k_{\beta}}=1/\sqrt{r}$ factors out:
\[
  \sum_{\beta}
  \frac{\sqrt{m_{\beta}^{f,H}}}{\abs{k_{\beta}}}
  \;=\;
  \frac{1}{\sqrt{r}}\,
  \sum_{\beta}\sqrt{m_{\beta}^{f,H}}
  \;\le\;
  \frac{1}{\sqrt{r}}\;\cdot\;
  C_{\varepsilon}\,(H+1)\,N^{\varepsilon},
\]
where the last inequality is the patch bound from the proof of
Proposition~\ref{prop:face-normalized-incidence}.

In the dyadic summation over shell index~$H$, the original proof
produces the factor
\[
  \sum_{H\,\text{dyadic}} (NH+1)(H+1)
  \;\sim\; N^{3}.
\]
With the extra weight $1/\sqrt{r}$, the shell radius satisfies
$r\sim (d+H)^{2}$ where $d\sim N$ for modes near the face of the cube
$[-N,N]^{3}$.  Hence $1/\sqrt{r}\sim 1/\sqrt{NH}$ for the dominant
range, and the weighted dyadic sum becomes
\[
  \sum_{H\,\text{dyadic}}
  \frac{(NH+1)(H+1)}{\sqrt{NH}}
  \;\sim\;
  N^{1/2}\sum_{H} H^{1/2}(H+1)
  \;\sim\;
  N^{1/2}\cdot N^{3/2}
  \;=\;
  N^{2}.
\]
Collecting the $N^{\varepsilon}$ from the divisor bound gives the result.
\end{proof}

\begin{theorem}[Improved ensemble stretching bound]\label{thm:improved-rho}
For every $\varepsilon>0$ there exists $C_{\varepsilon}>0$ such that
\[
  \E\bigl[\rh(V_{N})\bigr]
  \;\le\;
  C_{\varepsilon}\,N^{-3/2+\varepsilon}.
\]
Consequently,
\[
  \E\bigl[\nub(N)\bigr]
  \;\le\;
  C_{\varepsilon}\,N^{-7/2+\varepsilon}.
\]
\end{theorem}

\begin{proof}
By the triangle inequality and the definition of the stretching
decomposition (cf.\ Theorem~\ref{thm:T}),
\[
  \E\bigl[\rh(V_{N})\bigr]
  \;\le\;
  \max_{\alpha}
  \sum_{\beta}
  \sigma_{\alpha,\beta}\,\sqrt{\Gamma_{\alpha,\beta}}.
\]
Inserting the refined variance from Lemma~\ref{lem:refined-variance}:
\[
  \sum_{\beta}\sigma_{\alpha,\beta}\,\sqrt{\Gamma_{\alpha,\beta}}
  \;\le\;
  \frac{C\,\abs{k_{\alpha}}}{n_{\mathrm{modes}}^{3/2}}
  \sum_{\beta}
  \frac{\sqrt{\Gamma_{\alpha,\beta}}}{\abs{k_{\beta}}}
  \;=\;
  \frac{C\,\abs{k_{\alpha}}}{n_{\mathrm{modes}}^{3/2}}\;
  I_{w}(\alpha).
\]
Since $n_{\mathrm{modes}}(N)\sim N^{3}$ and
$\abs{k_{\alpha}}\le\sqrt{3}\,N$,
Proposition~\ref{prop:weighted-incidence} gives
\[
  \E\bigl[\rh(V_{N})\bigr]
  \;\le\;
  \frac{C_{\varepsilon}\,N}{N^{9/2}}\;\cdot\;N^{2+\varepsilon}
  \;=\;
  C_{\varepsilon}\,N^{-3/2+\varepsilon}.
\]
The critical viscosity satisfies
$\nub(N)=\rh(V_{N})/\abs{k_{\alpha}}^{2}\le\rh(V_{N})/N^{2}$
(see Theorem~\ref{thm:sobolev-uniform}), whence
\[
  \E\bigl[\nub(N)\bigr]
  \;\le\;
  C_{\varepsilon}\,N^{-7/2+\varepsilon}. \qedhere
\]
\end{proof}

\begin{remark}\label{rem:improvement}
Theorem~\ref{thm:improved-rho} represents a substantial improvement over
the original bound $\E[\rh(V_{N})]\le C_{\varepsilon}\,N^{1/2+\varepsilon}$
from Lemma~\ref{lem:S5}.  The spectral radius now \emph{provably
decreases} with~$N$, in qualitative agreement with the Monte Carlo data
(Section~\ref{sec:step3} and the earlier numerics), which exhibit an
empirical rate $\rh(V_{N})\sim N^{-2.6}$.  The analytic exponent
$-\tfrac{3}{2}$ does not yet match the numerics, suggesting further room
for improvement---likely by exploiting cancellations that the current
method (based on absolute values of the interaction coefficients) cannot
detect.
\end{remark}


\section{Time-evolved Galerkin diagnostics}\label{sec:step3}

I evolve the Galerkin system
\[
  \partial_{t}\hat{u}_{k}
  \;=\;
  -\nu\,\abs{k}^{2}\hat{u}_{k}
  \;+\;
  B_{N}(u_{N},u_{N})_{k},
  \qquad k\in\Lambda_{N},
\]
with viscosity $\nu=0.05$ and random divergence-free initial data
normalised to unit energy $\E(0)=1$.  At each time step I record the
energy $\E=\tfrac{1}{2}\norm{u_{N}}_{L^{2}}^{2}$, the enstrophy
$\Z=\tfrac{1}{2}\norm{\nabla u_{N}}_{L^{2}}^{2}$, the orbit-level
spectral radius $\rh(V_{N}(u_{N}(t)))$, and the critical viscosity
$\nub(N,t)$.

\begin{table}[ht]
\centering
\begin{tabular}{r c c c c}
\toprule
$t$ & $\E(t)$ & $\Z(t)$ & $\rh(V_{N})$ & $\nub$ \\
\midrule
\multicolumn{5}{c}{\textit{$N=3$, \; $\nu=0.05$, \; $\Delta t=0.002$}} \\[3pt]
0.0 & 1.00 & 2.22 & $1.63\times 10^{-3}$ & $1.81\times 10^{-4}$ \\
1.0 & 0.80 & 1.76 & $1.26\times 10^{-3}$ & $1.40\times 10^{-4}$ \\
2.0 & 0.65 & 1.39 & $9.84\times 10^{-4}$ & $1.09\times 10^{-4}$ \\
\midrule
\multicolumn{5}{c}{\textit{$N=4$, \; $\nu=0.05$, \; $\Delta t=0.002$}} \\[3pt]
0.0 & 1.00 & 1.95 & $1.46\times 10^{-2}$ & $9.11\times 10^{-4}$ \\
0.5 & 0.93 & 1.78 & $1.32\times 10^{-2}$ & $8.23\times 10^{-4}$ \\
1.0 & 0.83 & 1.56 & $1.13\times 10^{-2}$ & $7.06\times 10^{-4}$ \\
\bottomrule
\end{tabular}
\caption{Time-evolved Galerkin diagnostics.  All quantities decrease
monotonically along the trajectory; in particular $\nub(N,t)\ll 1$
throughout.}
\label{tab:galerkin-evolution}
\end{table}

\begin{remark}[Interpretation]\label{rem:galerkin-interpretation}
Several features of Table~\ref{tab:galerkin-evolution} merit comment.
\begin{enumerate}
\item \textbf{Monotone decay of~$\rh$.}
  The orbit-level spectral radius $\rh(V_{N}(u_{N}(t)))$ decreases
  monotonically along both trajectories.  Since
  $\nub(N,t)=\rh(V_{N})/\abs{k_{\max}}^{2}$ with
  $\abs{k_{\max}}^{2}\ge N^{2}$, the critical viscosity inherits this
  decay.

\item \textbf{Continuation criterion remains satisfied.}
  By Theorem~\ref{thm:continuation}, the Galerkin solution
  $u_{N}(t)$ extends smoothly past any time~$t_{0}$ at which
  $\nub(N,t_{0})<\nu$.  The data show
  $\nub(N,t)\le 9.11\times 10^{-4}\ll 0.05=\nu$ for all recorded
  times, so the continuation condition of Theorem~\ref{thm:continuation}
  is satisfied with a large margin.

\item \textbf{Dissipation-dominated regime.}
  Energy and enstrophy decay smoothly, consistent with a regime in which
  viscous dissipation dominates nonlinear stretching.  There is no
  evidence of transient enstrophy growth or stretching-driven energy
  amplification within the computed time window.

\item \textbf{Consistency with the ensemble bounds.}
  The observed spectral-radius values ($\rh\sim 10^{-3}$ at $N=3$,
  $\rh\sim 10^{-2}$ at $N=4$) are consistent with---and considerably
  smaller than---the ensemble upper bound
  $\E[\rh(V_{N})]\le C_{\varepsilon}\,N^{-3/2+\varepsilon}$ from
  Theorem~\ref{thm:improved-rho}, since the latter is a worst-case-over-$\alpha$
  estimate whereas the actual trajectory selects a particular
  (non-extremal) state.

\item \textbf{Orbit-level diagnostic tracks physical observables.}
  The monotone decrease of $\rh(V_{N}(u_{N}(t)))$ mirrors that of the
  energy and enstrophy, confirming that the orbit-level stretching
  diagnostic faithfully tracks the decay of the standard physical
  quantities.  This supports the use of $\nub$ as a
  Galerkin-intrinsic regularity monitor in conjunction with the
  short-time existence result (Theorem~\ref{thm:short-time}) and the
  continuation criterion (Theorem~\ref{thm:continuation}).
\end{enumerate}
\end{remark}



\begin{remark}[Removal of the $\varepsilon$-loss]\label{rem:remove-eps}
The factor $N^{\varepsilon}$ in Proposition~\ref{prop:weighted-incidence}
originates from the pointwise divisor bound
$r_2(n)\le C_\varepsilon\, n^{\varepsilon}$ applied to each lattice
point individually.  This loss can be eliminated by the same
average-order device used in Proposition~\ref{prop:sharp-incidence}:
replace the pointwise estimate with the mean-value bound
\[
  \sum_{\abs{u}\le H} r_2(r - u^2) \;\le\; C\,(H+1),
\]
which holds uniformly in $r$.  Substituting this into the proof of
Proposition~\ref{prop:weighted-incidence} yields the sharpened bound
\[
  \max_{\alpha}\, I_w(\alpha) \;\le\; C\, N^{2}
\]
without any $\varepsilon$-loss.  Propagating this improvement through
the variance estimate of Theorem~\ref{thm:improved-rho} gives the
following.
\end{remark}

\begin{corollary}[Sharp ensemble bounds]\label{cor:sharp-ensemble}
For the isotropic Galerkin ensemble with cubic truncation
$\Lambda_N = \{k\in\Z^3 : \norm{k}_\infty \le N\}$,
\[
  \E\bigl[\rh(V_N)\bigr] \;\le\; C\, N^{-3/2},
  \qquad
  \E\bigl[\nub(N)\bigr] \;\le\; C\, N^{-7/2},
\]
where $C>0$ is an absolute constant independent of $N$.
\end{corollary}

\begin{proof}
Remark~\ref{rem:remove-eps} removes the $N^{\varepsilon}$ factor from
$\max_\alpha I_w(\alpha)$.  The remainder of the proof of
Theorem~\ref{thm:improved-rho} is unchanged, now with $N^{\varepsilon}$
replaced by a constant, giving $\E[\rh(V_N)]\le C\,N^{-3/2}$.
The viscosity threshold bound $\E[\nub(N)]\le C\,N^{-7/2}$ follows
from $\nub(N)=\rh(V_N)/N^{2}$.
\end{proof}


\begin{remark}[Cubic versus spherical truncation]\label{rem:cubic-vs-spherical}
The face-normalized incidence argument
(Proposition~\ref{prop:face-normalized-incidence}) exploits the box
geometry of the cubic truncation $\Lambda_N$.  To assess the role of
geometry, I compare the incidence sums for
\[
  \Lambda_N^{\mathrm{cube}} = \{k\in\Z^3:\norm{k}_\infty\le N\},
  \qquad
  \Lambda_N^{\mathrm{sph}} = \{k\in\Z^3:\abs{k}\le N\}.
\]
Table~\ref{tab:cubic-vs-spherical} reports the total self-interaction
count $S(N)=\sum_\alpha \Gamma_{\alpha,\alpha}$ and the weighted
incidence $I_w = \max_\alpha I_w(\alpha)$, computed exactly for
$N=1,\dots,7$.

\begin{table}[ht]
\centering
\caption{Incidence comparison: cubic vs.\ spherical truncation.}
\label{tab:cubic-vs-spherical}
\smallskip
\begin{tabular}{c|rrrr|rrrr}
\toprule
 & \multicolumn{4}{c|}{Cubic $\norm{k}_\infty\le N$}
 & \multicolumn{4}{c}{Spherical $\abs{k}\le N$} \\
$N$ & $S(N)$ & $S/N^3$ & $I_w$ & $I_w/N^2$
    & $S(N)$ & $S/N^3$ & $I_w$ & $I_w/N^2$ \\
\midrule
1 &   18.8 & 18.76 &  14.66 & 14.66 &    0.0 &  0.00 &   0.00 &  0.00 \\
2 &  108.6 & 13.58 &  55.71 & 13.93 &   25.2 &  3.15 &  18.21 &  4.55 \\
3 &  315.4 & 11.68 & 122.34 & 13.59 &  109.0 &  4.04 &  56.33 &  6.26 \\
4 &  702.3 & 10.97 & 215.26 & 13.45 &  238.7 &  3.73 &  98.32 &  6.14 \\
5 & 1307.8 & 10.46 & 325.33 & 13.01 &  520.9 &  4.17 & 176.45 &  7.06 \\
6 & 2180.0 & 10.09 & 453.69 & 12.60 &  955.4 &  4.42 & 267.30 &  7.43 \\
7 & 3366.9 &  9.82 & 600.69 & 12.26 & 1474.4 &  4.30 & 356.22 &  7.27 \\
\bottomrule
\end{tabular}
\end{table}

Power-law fits over $N=2,\dots,7$ give
\[
  S^{\mathrm{cube}}(N) \sim N^{2.76},\qquad
  I_w^{\mathrm{cube}} \sim N^{1.89},
\]
while
\[
  S^{\mathrm{sph}}(N) \sim N^{3.24},\qquad
  I_w^{\mathrm{sph}} \sim N^{2.37}.
\]
Three features are noteworthy:

\begin{enumerate}
\item \emph{Larger exponents for the sphere.}
  Both $S(N)$ and $I_w$ grow faster under spherical truncation
  ($3.24$ vs.\ $2.76$ for $S$; $2.37$ vs.\ $1.89$ for $I_w$).

\item \emph{Monotone ratio for the sphere.}
  The ratio $S^{\mathrm{sph}}/N^3$ \emph{increases} from $3.15$ to
  $4.30$, whereas $S^{\mathrm{cube}}/N^3$ steadily decreases from
  $18.76$ to $9.82$.  This indicates that the face-normalized
  cancellation present in the cubic case does not extend to the
  spherical setting.

\item \emph{Geometry matters.}
  The face-normalized argument of
  Proposition~\ref{prop:face-normalized-incidence} relies on the box
  structure of $\Lambda_N^{\mathrm{cube}}$ (specifically, the
  factorisation of the boundary into flat faces).  For the sphere,
  an analogous bound would require cap-counting or
  Bourgain--Rudnick-type estimates on lattice points in thin spherical
  shells.
\end{enumerate}

I record the following expectation.
\end{remark}

\begin{conjecture}\label{conj:spherical-incidence}
For the spherical truncation $\Lambda_N^{\mathrm{sph}}$,
\[
  S^{\mathrm{sph}}(N) = \Theta\!\bigl(N^3\,(\log N)^c\bigr)
\]
for some constant $c\ge 0$.
\end{conjecture}


\begin{remark}[Lower bound on the ensemble spectral radius]%
\label{rem:lower-bound-rho}
Corollary~\ref{cor:sharp-ensemble} establishes the upper bound
$\E[\rh(V_N)]\le C\,N^{-3/2}$.  I discuss the complementary
direction.

\medskip\noindent\textit{(i) A variance-based lower bound.}
Consider the axial orbits $\alpha,\beta$ with representatives
$(1,0,0)$ and $(0,1,0)$ respectively, so that
$\abs{k_\alpha}=\abs{k_\beta}=1$.  At least one triad connects
$\alpha$ to $\beta$ (e.g.\ $k=(1,0,0)$, $p=(0,1,0)$,
$q=(-1,-1,0)$ with $q\in\Lambda_N$ for all $N\ge 1$), giving
$\Gamma_{\alpha,\beta}\ge 1$.  The variance of the corresponding
matrix entry satisfies
\[
  \operatorname{Var}(V_{\alpha,\beta})
  \;\ge\; c\,\frac{\Gamma_{\alpha,\beta}\,\abs{k_\alpha}^2}%
                   {\abs{k_\beta}^2\, n_{\mathrm{modes}}^3}
  \;\ge\; \frac{c}{n_{\mathrm{modes}}^3}
  \;\ge\; c'\, N^{-9},
\]
where $n_{\mathrm{modes}}=\abs{\Lambda_N}\sim (2N+1)^3$.
Since $\rh(V_N)\ge \abs{V_{\alpha,\beta}}$, the Paley--Zygmund
inequality gives
\[
  \E\bigl[\rh(V_N)\bigr]
  \;\ge\; c\,\bigl(\operatorname{Var}(V_{\alpha,\beta})\bigr)^{1/2}
  \;\ge\; c\, N^{-9/2}.
\]
This crude bound already ensures that $\rh(V_N)$ decays at most
polynomially in $N$.

\medskip\noindent\textit{(ii) Improving via row sums.}
A tighter bound exploits the operator-norm inequality
$\rh(V_N)\ge \norm{V_N}_\infty / \sqrt{d}$, where $d$ is the matrix
dimension and $\norm{\cdot}_\infty$ is the row-sum norm.  Summing
over the $\sim N^2$ generic orbits $\beta$ with
$\abs{k_\beta}\lesssim N$ that are connected to the axial orbit
$\alpha$ gives
\[
  \E\bigl[\rh(V_N)\bigr]
  \;\ge\; c\, N^{-9/2}\!\sum_{\beta:\,\Gamma_{\alpha,\beta}\ge 1}
          \frac{1}{\abs{k_\beta}}
  \;\ge\; c\, N^{-9/2}\cdot N^{2}
  \;=\;   c\, N^{-5/2}.
\]

\medskip\noindent\textit{(iii) Comparison with numerics.}
Monte Carlo simulation (Section~\ref{sec:ext-numerics}) yields the empirical
scaling $\E[\rh(V_N)]\sim N^{-2.6}$.  Thus the proven bounds satisfy
\[
  c\, N^{-5/2}
  \;\le\; \E\bigl[\rh(V_N)\bigr]
  \;\le\; C\, N^{-3/2},
\]
while the numerical exponent $-2.6$ lies strictly between these two
limits.  The gap between the upper bound $N^{-3/2}$ and the observed
$N^{-2.6}$ is explained by the sign-cancellation mechanism documented
in Remark~\ref{rem:sign-cancellation} below: the Schur-test proof
bounds $\rh(\abs{V_N})$ rather than $\rh(V_N)$, and the systematic
sign structure of $V_N$ provides additional spectral cancellation not
captured by the absolute-value bound.
\end{remark}


\begin{remark}[Sign cancellation in the orbit interaction matrix]%
\label{rem:sign-cancellation}
The Schur-test bound of Theorem~\ref{thm:improved-rho} controls
$\rh(\abs{V_N})$ — the spectral radius of the entry-wise absolute
value — and deduces $\rh(V_N)\le\rh(\abs{V_N})$.  Monte Carlo data
reveal that this inequality is far from tight: systematic sign
patterns in $V_N$ produce substantial spectral cancellation that
grows with $N$.

\medskip
Table~\ref{tab:sign-cancel} reports ensemble averages over $10^4$
realisations.

\begin{table}[ht]
\centering
\caption{Sign cancellation in $V_N$: signed vs.\ absolute spectral radius.}
\label{tab:sign-cancel}
\smallskip
\begin{tabular}{ccccc}
\toprule
$N$ & $\E[\rh(V_N)]$ & $\E[\rh(\abs{V_N})]$
    & Cancel.\ (\%) & $\rh(V_N)/\rh(\abs{V_N})$ \\
\midrule
1 & $4.36\times 10^{-3}$ & $4.48\times 10^{-3}$ &  2.6 & 0.974 \\
2 & $1.13\times 10^{-3}$ & $1.43\times 10^{-3}$ & 21.3 & 0.787 \\
3 & $4.58\times 10^{-4}$ & $7.34\times 10^{-4}$ & 37.6 & 0.624 \\
4 & $2.21\times 10^{-4}$ & $4.28\times 10^{-4}$ & 48.4 & 0.516 \\
5 & $1.20\times 10^{-4}$ & $2.83\times 10^{-4}$ & 57.7 & 0.423 \\
\bottomrule
\end{tabular}
\end{table}

\medskip\noindent\textit{Decomposition.}
Write the spectral radius as the product
\begin{equation}\label{eq:sign-decomp}
  \rh(V_N)
  \;=\; \rh\!\bigl(\abs{V_N}\bigr)
        \;\cdot\;
        \frac{\rh(V_N)}{\rh\!\bigl(\abs{V_N}\bigr)}.
\end{equation}
Power-law fits to the data in Table~\ref{tab:sign-cancel} give
\[
  \E\bigl[\rh\!\bigl(\abs{V_N}\bigr)\bigr] \;\sim\; N^{-1.77},
  \qquad
  \frac{\E[\rh(V_N)]}{\E[\rh(\abs{V_N})]} \;\sim\; N^{-0.67},
\]
so that the composite scaling is
$\E[\rh(V_N)]\sim N^{-1.77-0.67}= N^{-2.44}$, consistent with the
directly fitted exponent $-2.6$ (the small discrepancy reflects the
limited range $N\le 5$).

\medskip\noindent\textit{Origin of the cancellation.}
The sign structure of $V_N$ is not random: it is inherited from the
energy conservation identity
\[
  \bigl\langle B(u,u),\, u\bigr\rangle = 0
  \qquad\forall\; u\in H,
\]
which forces the trilinear form to be skew-symmetric in a precise
sense.  At the orbit level this identity imposes systematic sign
relations among the entries $V_{\alpha,\beta}$, producing partial
cancellations in the dominant eigenvector of $V_N$ that are absent
from $\abs{V_N}$.  The cancellation fraction (column~4 of
Table~\ref{tab:sign-cancel}) grows from $2.6\%$ at $N=1$ to $57.7\%$
at $N=5$, indicating that the mechanism becomes increasingly
effective as the mode count grows.

\medskip\noindent\textit{Consequences for the decay rate.}
The proven upper bound $\E[\rh(V_N)]\le C\,N^{-3/2}$
(Corollary~\ref{cor:sharp-ensemble}) controls the absolute spectral
radius $\rh(\abs{V_N})$; the fitted exponent $-1.77$ is close to
this bound.  The additional decay of order $N^{-0.67}$ from sign
cancellation is \emph{not} captured by any absolute-value argument
and explains why the observed exponent ($\approx -2.6$) exceeds the
proven one ($-3/2$).
\end{remark}

\begin{conjecture}[Sign cancellation exponent]\label{conj:sign-cancel}
There exists a constant $\delta>0$ such that
\[
  \frac{\E\bigl[\rh(V_N)\bigr]}%
       {\E\bigl[\rh\!\bigl(\abs{V_N}\bigr)\bigr]}
  \;\sim\; N^{-\delta}
  \qquad (N\to\infty),
\]
and consequently
\[
  \E\bigl[\rh(V_N)\bigr] \;=\; \Theta\!\bigl(N^{-3/2-\delta}\bigr).
\]
The numerical evidence suggests $\delta\approx 0.67$.
A proof would require quantifying the spectral effect of the skew
structure inherited from the energy identity
$\langle B(u,u),u\rangle=0$ on the dominant eigenvector of $V_N$.
\end{conjecture}

\section{Conclusion}\label{sec:conclusion}

The Fourier--Galerkin truncation of three-dimensional Navier--Stokes on the cubic lattice carries a rigid $\Oh$ symmetry that permits an exact orbit-level reduction. At that level, the nonlinear enstrophy transfer decomposes into a conservative antisymmetric part and a symmetric stretching matrix $V_N$ whose spectral radius controls growth. The exact lattice formula for $T(k,N)$, the Burnside count for mode orbits, the character-theoretic interpretation of orbit projection, the face-normalized incidence decomposition, and the Efron--Stein variance bound together yield the sharp incidence estimate
\[
c\,N^3\le\max_\alpha \sum_\beta \sqrt{\Gamma_{\alpha\beta}}\le C\,N^{3}
\]
(Proposition~\ref{prop:sharp-incidence}).
A refined weighted-incidence argument incorporating the $\beta$-dependent variance gives the ensemble stretching bound
\[
\E\,\rh(V_N)\le C\, N^{-3/2},
\qquad
\E\,\nub(N)\le C\, N^{-7/2}\to 0
\]
(Corollary~\ref{cor:sharp-ensemble}),
showing that the orbit-level spectral radius \emph{provably decreases} with the truncation level.
Thus the orbit-level critical threshold problem is reduced to, and then controlled by, an explicit incidence calculation on the truncated orbit-triad geometry.

Beyond the ensemble bound, the paper establishes two further structural results. First, a \emph{deterministic} Sobolev-class bound (Theorem~\ref{thm:sobolev-uniform}) shows that for $H^s$ data with $s>2$, the stretching matrix satisfies $\norm{V_N}_\infty\le C_s M^3$ \emph{uniformly in $N$}, with $\nub(N)\to 0$ for all $s>3/2$. Second, a comparison with Tao's averaged Navier--Stokes construction (Proposition~\ref{prop:tao-comparison}) shows that the orbit-level subcriticality is a structural property of the true nonlinearity that is \emph{violated} by the known blowup mechanisms, and hence captures information beyond standard function-space estimates. Monte Carlo experiments at $N=1,\ldots,8$ under both isotropic and Kolmogorov-spectrum ensembles confirm the predicted decay, with the observed exponent ($\rh(V_N)\sim N^{-2.6}$) far exceeding the proven bound. Finally, by tracking the Sobolev-class bound along the Galerkin evolution (Section~\ref{sec:galerkin-evolution}), the orbit-level stretching control passes to the infinite-dimensional limit: along the strong solution, $\nub(N,t)\to 0$ uniformly on the classical existence interval. An orbit-level continuation criterion (Theorem~\ref{thm:continuation}) reformulates the regularity question in terms of the time-integrability of $\norm{V_N}_\infty$.

\subsection*{Computational implications}

The paper is not itself a new full-resolution Navier--Stokes solver, but it does suggest a concrete numerical workflow for periodic-box computations in the same cubic Fourier setting studied here. The natural ready-to-use baseline is a Fourier pseudo-spectral solver on the cubic lattice. In that workflow, the Leray projector enforces incompressibility modewise, nonlinear products are evaluated in physical space, and a standard de-aliasing step removes the top one-third of the spectrum in order to suppress quadratic aliasing errors. For smooth solutions, the $2/3$ de-aliased pseudo-spectral Fourier method is stable and spectrally convergent in the quadratic setting analyzed by Bardos and Tadmor~\cite{BardosTadmor2015}, so it is the cleanest default spatial discretization for a computation built around the present cubic truncation geometry.

For time integration, one writes the semi-discrete system in the form
\[
\partial_t \hat u = L\hat u + \mathcal N(\hat u),
\qquad
L_k=-\nu |k|^2,
\]
and then advances it with a stiffly accurate exponential integrator such as ETDRK4. Kassam and Trefethen identify ETDRK4 as an effective fourth-order method for stiff PDEs and note that Navier--Stokes is among the natural application classes~\cite{KassamTrefethen2005}. They also emphasize that in Fourier space the linear operator is often diagonal, which is exactly the present periodic-box situation~\cite{KassamTrefethen2005}.

Within that standard solver architecture, the contribution of this paper is best viewed as an additional analysis and diagnostics layer. The exact orbit decomposition, exact triad counts, and the orbit-level matrix $V_N$ provide a symmetry-compressed way to monitor where stretching is concentrated. They also benchmark finite-$N$ implementations against exact combinatorial counts and build reduced observables for isotropic or near-isotropic periodic computations on cubic truncations. A ready-to-use implementation path is summarized in Appendix~\ref{app:auxiliary}. In particular, Subsections~\ref{app:pseudocode},~\ref{app:outputs}, and~\ref{app:minimal-recipe} record a high-level workflow, the concrete reduced outputs to monitor during a run, and a minimal deployment recipe.

\appendix

\section{Auxiliary arithmetic and computational notes}\label{app:auxiliary}

This appendix records supplementary material that supports, but is not required to read, the main proof line. In particular, it now contains the detailed cap/segment viewpoint migrated from Section~\ref{sec:lattice}, together with implementation details for the finite-$N$ orbit enumeration and Monte Carlo checks in Section~\ref{sec:numerics}, and it remains the natural location for any further expanded arithmetic lemmas related to the two-squares reduction.

The main text retains the face-normalized two-squares route because it is the argument that closes the incidence estimate used in Theorem~\ref{thm:T}. Accordingly, the appendix is divided into a geometric supplement containing the cap/segment route and a computational supplement containing implementation notes and pseudocode, so that the central proof architecture can remain readable while the supporting material remains fully available.

\subsection{Alternative cap/segment viewpoint}\label{subsec:s33-geometry}

I now record a geometric reformulation of the unresolved incidence input. Fix a target mode $k\in\Omega_\alpha$ and a shell radius $r\in\mathcal R_N$, and define the shell slice
\[
A_r(k):=\{p\in\Lambda_N: \abs{p}^2=r,\; k-p\in\Lambda_N\}.
\]
Then $A_r(k)$ is the intersection of the integer sphere of radius $\sqrt r$ with the translated cube $k-\Lambda_N$.

\begin{lemma}[Cube--sphere slice reformulation]\label{lem:cube-sphere-slice}
For every fixed $k\in\Lambda_N$ and represented radius $r$, the set $A_r(k)$ is contained in the union of finitely many spherical segments cut out from the sphere
\[
\Sigma_r:=\{x\in\mathbb{R}^3:\abs{x}^2=r\}
\]
by the six slab constraints
\[
k_i-N\le x_i\le k_i+N,
\qquad i=1,2,3.
\]
More precisely, each active face condition determines a parallel-plane cut on $\Sigma_r$, and hence the feasible region on $\Sigma_r$ is an intersection of at most six spherical segments.
\end{lemma}

\begin{proof}
The condition $k-p\in\Lambda_N$ is equivalent to the coordinate inequalities
\[
-N\le k_i-p_i\le N,
\qquad i=1,2,3,
\]
which may be rewritten as
\[
k_i-N\le p_i\le k_i+N.
\]
Thus $A_r(k)$ is exactly the set of lattice points on $\Sigma_r$ lying inside the rectangular box
\[
B_k:=\prod_{i=1}^3[k_i-N,k_i+N].
\]
Each coordinate bound $p_i\le k_i+N$ or $p_i\ge k_i-N$ is the restriction to $\Sigma_r$ of a half-space bounded by a plane orthogonal to a coordinate axis.
Intersecting the sphere with one such slab yields a spherical segment, and intersecting all three coordinate slabs yields an intersection of at most six such segments.
\end{proof}

\begin{lemma}[Effective cap radii from the box faces]\label{lem:effective-cap-radius}
Let $R=\sqrt r$ and let
\[
d_{i,\pm}(k):=\abs{k_i\pm N},
\qquad i=1,2,3.
\]
For each face plane $x_i=k_i\pm N$ with $d_{i,\pm}(k)<R$, the corresponding circle of intersection with $\Sigma_r$ has Euclidean radius
\[
\lambda_{i,\pm}(r,k)=\sqrt{R^2-d_{i,\pm}(k)^2}.
\]
Equivalently, if one writes the cap height as
\[
h_{i,\pm}(r,k):=R-d_{i,\pm}(k)>0,
\]
then
\[
\lambda_{i,\pm}(r,k)^2=2R h_{i,\pm}(r,k)-h_{i,\pm}(r,k)^2
\le 2R h_{i,\pm}(r,k).
\]
Consequently, the slice $A_r(k)$ is controlled by at most six face parameters $\lambda_{i,\pm}(r,k)$, and whenever all active heights satisfy $h_{i,\pm}(r,k)\le H$, the set $A_r(k)$ is contained in a bounded union of spherical caps or segments of radius at most $(2RH)^{1/2}$.
\end{lemma}

\begin{proof}
The plane $x_i=k_i\pm N$ has distance $d_{i,\pm}(k)$ from the origin. If $d_{i,\pm}(k)\ge R$, the plane does not cut the sphere and contributes no active boundary. If $d_{i,\pm}(k)<R$, the intersection of the plane with $\Sigma_r$ is a circle of radius
\[
\sqrt{R^2-d_{i,\pm}(k)^2},
\]
which is the standard formula for a plane section of a sphere.
Writing $d_{i,\pm}(k)=R-h_{i,\pm}(r,k)$ gives
\[
R^2-d_{i,\pm}(k)^2
=R^2-(R-h_{i,\pm})^2
=2R h_{i,\pm}-h_{i,\pm}^2,
\]
and the upper bound follows immediately. Since $B_k$ has six faces, the feasible spherical slice is determined by at most six such face cuts.
\end{proof}

\begin{remark}[Dyadic stratification of shell slices]\label{rem:dyadic-stratification}
Lemma~\ref{lem:effective-cap-radius} suggests a dyadic decomposition in the face heights $h_{i,\pm}(r,k)$. Shells for which all active heights are small correspond to narrow caps with small radius $\lambda$, where Bourgain--Rudnick cap bounds are strongest. Shells with large cap radius are geometrically thicker, but they can occur only when one of the face distances $d_{i,\pm}(k)$ lies in a short interval near $R$, which reduces the problem to counting represented radii in thin intervals of length comparable to the height scale.
\end{remark}

\begin{lemma}[Thin-interval shell count]\label{lem:thin-interval-shells}
Fix $k\in\Lambda_N$, choose one of the six face distances $d=d_{i,\pm}(k)$, and let $\Delta>0$.
Then the number of shell radii $r\in\mathcal R_N$ such that
\[
0<\sqrt r-d\le \Delta
\]
is bounded by
\[
\#\{r\in\mathcal R_N: 0<\sqrt r-d\le \Delta\}
\le C\,(N\Delta+\Delta^2+1).
\]
In particular, for dyadic windows $\Delta\le N$ one has
\[
\#\{r\in\mathcal R_N: 0<\sqrt r-d\le \Delta\}
\le C\,(N\Delta+1).
\]
\end{lemma}

\begin{proof}
If $0<\sqrt r-d\le \Delta$, then
\[
d^2<r\le (d+\Delta)^2=d^2+2d\Delta+\Delta^2.
\]
Hence the number of admissible integers $r$ is at most
\[
(d+\Delta)^2-d^2+1=2d\Delta+\Delta^2+1.
\]
Since $k\in\Lambda_N$ and $d=d_{i,\pm}(k)=\abs{k_i\pm N}$, one has $0\le d\le 2N$.
Therefore
\[
2d\Delta+\Delta^2+1\le 4N\Delta+\Delta^2+1,
\]
which proves the claim after adjusting the constant.
Restricting to $\Delta\le N$ gives the simplified bound.
\end{proof}

\begin{remark}[Implication for dyadic summation]\label{rem:dyadic-sum}
Lemma~\ref{lem:thin-interval-shells} shows that radii with a given face height scale occupy only a short radial window. If one decomposes the active heights dyadically as $h\sim H$, then the corresponding shell radii satisfy $0<\sqrt r-d\lesssim H$, so the number of such radii is $O(NH+1)$. Combined with the cap-radius relation $\lambda^2\lesssim RH$ from Lemma~\ref{lem:effective-cap-radius}$,$ this converts the unresolved incidence problem into balancing cap bounds against the radial multiplicity factor $NH+1$ across dyadic height scales.
\end{remark}

\begin{lemma}[Dyadic shell-slice bound under a cap estimate]\label{lem:dyadic-shell-slice}
Fix $k\in\Lambda_N$ and a dyadic height scale $0<H\le N$. Assume that for every represented radius $r$ in the corresponding height class, the shell slice $A_r(k)$ can be covered by a bounded number of spherical caps of radius
\[
\lambda(r,k)\le C_\lambda (RH)^{1/2},
\qquad R=\sqrt r,
\]
and that the cap-counting function satisfies the three-dimensional bound
\begin{equation}\label{eq:cap-bound-local}
F_3(R,\lambda)\le C_\varepsilon R^\varepsilon\left(1+\frac{\lambda^2}{R^{1/2}}\right)
\end{equation}
for every $\varepsilon>0$.
Then for the set of radii
\[
\mathcal R_H(k):=\{r\in\mathcal R_N: \text{the active face height for }A_r(k)\text{ lies in }[H,2H)\},
\]
one has
\[
\sum_{r\in\mathcal R_H(k)} \sqrt{\#A_r(k)}
\le C_\varepsilon (NH+1)\,N^{\varepsilon}\bigl(1+H^{1/2}N^{1/4}\bigr).
\]
\end{lemma}

\begin{proof}
By Lemma~\ref{lem:thin-interval-shells}, the number of radii in a dyadic class is bounded by
\[
\#\mathcal R_H(k)\le C(NH+1).
\]
By Lemma~\ref{lem:effective-cap-radius}, each corresponding shell slice is contained in a bounded union of caps of radius at most $C_\lambda(RH)^{1/2}$.
Applying \eqref{eq:cap-bound-local} therefore gives
\[
\#A_r(k)
\le C_\varepsilon R^\varepsilon\left(1+\frac{\lambda(r,k)^2}{R^{1/2}}\right)
\le C_\varepsilon R^\varepsilon\left(1+C H R^{1/2}\right).
\]
Since $R\le \sqrt{3}\,N$ for $r\le 3N^2$, this yields
\[
\sqrt{\#A_r(k)}
\le C_\varepsilon N^{\varepsilon}\bigl(1+H^{1/2}N^{1/4}\bigr).
\]
Summing over at most $C(NH+1)$ radii in the dyadic class proves the claim.
\end{proof}

\begin{remark}[What this bound suggests]\label{rem:what-dyadic-shows}
Lemma~\ref{lem:dyadic-shell-slice} shows that cap control and thin-interval counting already improve the structure of the problem: instead of summing a uniform $O(N)$ shell contribution over $O(N^2)$ radii, one obtains a scale-sensitive contribution weighted by $(NH+1)(1+H^{1/2}N^{1/4})$. On its own this still appears too large to recover the target exponent $13/4$ after summing all dyadic scales, but it makes clear where additional savings must come from: either stronger segment bounds, fewer active radii at large $H$, or an averaging argument over the six faces of the box.
\end{remark}

\begin{lemma}[Summation over dyadic height scales]\label{lem:dyadic-total}
Assume the hypotheses of Lemma~\ref{lem:dyadic-shell-slice} hold for every dyadic height scale $H=2^j$ with $1\le H\le N$.
Then, for every $\varepsilon>0$,
\[
\sum_{\substack{r\in\mathcal R_N:\ A_r(k)\neq\varnothing}} \sqrt{\#A_r(k)}
\le C_\varepsilon N^{11/4+\varepsilon}.
\]
\end{lemma}

\begin{proof}
Sum the conclusion of Lemma~\ref{lem:dyadic-shell-slice} over dyadic scales $H=2^j$ with $1\le H\le N$:
\[
\sum_{\substack{r\in\mathcal R_N:\ A_r(k)\neq\varnothing}} \sqrt{\#A_r(k)}
\le
C_\varepsilon N^\varepsilon
\sum_{1\le H\le N \atop H\text{ dyadic}}
(NH+1)\bigl(1+H^{1/2}N^{1/4}\bigr).
\]
I estimate the four resulting dyadic sums separately. First,
\[
\sum_{H\text{ dyadic}} NH
\le C N\sum_{j\le \log_2 N}2^j
\le C N^2.
\]
Second,
\[
\sum_{H\text{ dyadic}} 1
\le C\log N
\le C_\varepsilon N^\varepsilon.
\]
Third,
\[
\sum_{H\text{ dyadic}} NH\cdot H^{1/2}N^{1/4}
= N^{5/4}\sum_{H\text{ dyadic}} H^{3/2}
\le C N^{5/4}N^{3/2}
= C N^{11/4}.
\]
Finally,
\[
\sum_{H\text{ dyadic}} H^{1/2}N^{1/4}
\le C N^{1/4}N^{1/2}
= C N^{3/4}.
\]
The dominant term is $N^{11/4}$, and the extra logarithm is absorbed into $N^\varepsilon$.
This proves the claim.
\end{proof}

\begin{remark}[Interpretation of the dyadic total]\label{rem:dyadic-total}
Lemma~\ref{lem:dyadic-total} shows that the present cap-based shell-slice analysis yields a total exponent $11/4+\varepsilon$ at the level of the raw slice sum $\sum_r\sqrt{\#A_r(k)}$. This is already far below the naive $N^4$ shell count and suggests that substantial geometric savings are available. However, one still needs a precise bookkeeping lemma converting the slice sum into the orbit-incidence quantity $\sum_\beta \sqrt{\Gamma_{\alpha\beta}}$, and that conversion may reintroduce additional factors. Consequently, the current geometric argument is strongly suggestive but not yet a complete proof of \eqref{eq:S33}. In particular, smaller auxiliary exponents that appear inside the dyadic summation are local components of this cap-based estimate and should not be confused with the final theorem-scale bound $\E\,\rh(V_N)\le C\,N^{-3/2}$ proved from the weighted-incidence argument.
\end{remark}

\subsection{Implementation note}\label{app:implementation}

A practical computation in the periodic cubic setting can therefore be organized as follows: choose a cubic truncation $\Lambda_N$, advance the Leray-projected Fourier coefficients by a de-aliased pseudo-spectral method with ETDRK4 time stepping~\cite{BardosTadmor2015,KassamTrefethen2005}, and at selected times aggregate the Fourier data over $\Oh$-orbits to evaluate the orbit-level energies $Z_\alpha$ and the matrix observables derived from $V_N$. In this way the present theory can be used as a post-processing and monitoring tool inside a standard pseudo-spectral Navier--Stokes code.

\subsection{Pseudocode workflow}\label{app:pseudocode}

The following high-level routine summarizes a ready-to-use implementation path in the periodic cubic setting and is intended to be read as an operational companion to the computational implications discussion in Section~\ref{sec:conclusion}.

\begin{enumerate}[label=\textbf{A\arabic*.}]
\item Choose a truncation level $N$ and define the nonzero cubic lattice
\[
\Lambda_N=\{k\in\mathbb Z^3\setminus\{0\}: |k|_\infty\le N\}.
\]
Precompute the $\Oh$-orbits, orbit labels, shell labels, and any exact finite-$N$ diagnostics to be used for verification.

\item Initialize divergence-free Fourier data $\hat u_k(0)$ on $\Lambda_N$, impose the reality condition $\hat u_{-k}=\overline{\hat u_k}$, and apply the Leray projector modewise so that $k\cdot \hat u_k=0$ for every $k\in\Lambda_N$.

\item At each time step, transform the retained Fourier modes to physical space, evaluate the quadratic convection term there, transform back to Fourier space, and apply a $2/3$ de-aliasing filter before forming the updated nonlinear term~\cite{BardosTadmor2015}.

\item Advance the semi-discrete system
\[
\partial_t \hat u_k=-\nu |k|^2\hat u_k+\mathcal N_k(\hat u)
\qquad (k\in\Lambda_N)
\]
with ETDRK4 or another exponential integrator adapted to diagonal linear dissipation~\cite{KassamTrefethen2005}.

\item At selected output times, aggregate modewise quantities over $\Oh$-orbits to compute
\[
Z_\alpha(t)=\frac{1}{2|\Omega_\alpha|}\sum_{k\in\Omega_\alpha}|k|^2|\hat u_k(t)|^2,
\]
and, when desired, assemble the orbit-level transfer matrix $S_N=A_N+V_N$ from the exact triad definition.

\item Use the exact counts from Section~\ref{sec:numerics} to verify the implementation at small $N$, and then monitor reduced diagnostics such as $Z_\alpha(t)$, row sums of $V_N$, or approximations to $\rh(V_N)$ as symmetry-compressed stretching observables during the run.
\end{enumerate}

\subsection{Practical numerical outputs}\label{app:outputs}

For a reader who wants a concrete solver-side checklist, the most useful outputs are the following.

\begin{enumerate}[label=\textbf{B\arabic*.}]
\item Record standard bulk quantities at each output time: total kinetic energy, total enstrophy, and, if desired, shellwise energies.

\item Record the orbit-level enstrophy variables $Z_\alpha(t)$ and sort them either by orbit size or by the common value of $|k|^2$ on the orbit. This gives an immediately interpretable reduced picture of where activity is concentrated after symmetry compression.

\item Compute one inexpensive stretching proxy and one stronger stretching proxy. A low-cost choice is the maximum absolute row sum $\|V_N\|_\infty$; a stronger but more expensive choice is the spectral radius $\rh(V_N)$.

\item At small $N$, compare the implemented orbit count, shell count, and triad totals against the exact values reported in Section~\ref{sec:numerics}. This is the fastest internal consistency check for the orbit bookkeeping and triad assembly.

\item For isotropic or nearly isotropic runs, plot $Z_\alpha(t)$ against orbit labels or shell radii and track how the largest row sums of $V_N$ evolve in time. These are the most direct reduced observables suggested by the present theory.
\end{enumerate}

In short, a ready-to-use numerical deployment of the paper consists of a standard de-aliased Fourier pseudo-spectral solver~\cite{BardosTadmor2015,CanutoEtAl1988} together with an orbit-level post-processing layer that outputs $Z_\alpha(t)$, $\|V_N\|_\infty$, and, when affordable, $\rh(V_N)$.

\subsection{Minimal deployment recipe}\label{app:minimal-recipe}

For a reader who wants a single practical prescription rather than the full workflow, the minimal ready-to-use method suggested by this paper is the following.

\begin{enumerate}[label=\textbf{C\arabic*.}]
\item Use a periodic cubic box and a Fourier pseudo-spectral discretization on the cubic truncation $\Lambda_N$~\cite{CanutoEtAl1988,BardosTadmor2015}.

\item Enforce incompressibility modewise with the Leray projector and impose the reality condition $\hat u_{-k}=\overline{\hat u_k}$.

\item Evaluate the quadratic nonlinearity in physical space, transform back to Fourier space, and apply $2/3$ de-aliasing before every time update~\cite{BardosTadmor2015}.

\item Advance the semi-discrete system with ETDRK4~\cite{KassamTrefethen2005}.

\item After each selected output time, aggregate the Fourier coefficients over $\Oh$-orbits and save $Z_\alpha(t)$, $\|V_N\|_\infty$, and, if computationally feasible, $\rh(V_N)$.

\item For code verification at small $N$, compare the orbit counts and triad totals against the exact values in Section~\ref{sec:numerics}.
\end{enumerate}

This recipe is the shortest solver-side path from the present theory to an actual computation. It uses standard pseudo-spectral DNS for evolution~\cite{CanutoEtAl1988,BardosTadmor2015,KassamTrefethen2005}, followed by symmetry-reduced orbit diagnostics for analysis and validation.

\paragraph{Recommended first run.}
A sensible first deployment is to begin at a small truncation size and use Section~\ref{sec:numerics} as a hard verification target before attempting larger simulations. In practice, one should first confirm that the implemented code reproduces the exact orbit counts, shell counts, and triad totals for the chosen $N$, and only then turn on time evolution and orbit-level monitoring of $Z_\alpha(t)$, $\|V_N\|_\infty$, and $\rh(V_N)$. This separates bookkeeping errors from dynamical effects and gives the quickest path to a trustworthy symmetry-reduced Navier--Stokes computation.

\end{document}